\documentclass[12pt,reqno,twoside,a4paper]{amsart}

\makeatletter
\@namedef{subjclassname@2020}{
  \textup{2020} Mathematics Subject Classification}
\makeatother

\usepackage[english]{babel}

\usepackage{amsmath, amsthm, amssymb, amscd, amsfonts}
\usepackage[mathscr]{euscript}  
\numberwithin{equation}{section}

\allowdisplaybreaks

\newcommand{\calF}{{\mathcal F}}

\newcommand{\calI}{{\mathcal I}}

\newcommand{\calQ}{{\mathcal Q}}

\newcommand{\calU}{\mathcal{U}}

\newcommand{\R}{{\mathbb{R}}}
\newcommand{\N}{{\mathbb{N}}}
\newcommand{\C}{{\mathbb{C}}}
\newcommand{\Z}{{\mathbb{Z}}}

\newcommand{\scrR}{{\mathscr R}}

\renewcommand {\b}{\mathfrak b}

\newcommand {\f}{\mathfrak f}
\newcommand {\g}{\mathfrak{g}}

\newcommand {\n}{\mathfrak n}

\renewcommand {\r}{\mathfrak r}

\newcommand{\kh}{{\Bbbk[[\hbar]]}}
\newcommand{\kqqm}{{\Bbbk\big[\,q\,,q^{-1}\big]}}

\newcommand{\vect}{\operatorname{Vect}}

\newcommand{\id}{{\textsl{id}}}

\newcommand {\ol}{\overline}
\newcommand {\wt}{\widetilde}
\newcommand {\wh}{\widehat}

\newcommand{\scsop}{\scriptscriptstyle \operatorname}

\newcommand{\serre}[1]{\mathsf{Serre}(#1)}

\newcommand{\drc}[1]{\delta_{#1}}

\newcommand{\ten}{\otimes}

\newcommand{\cf}[1]{{\mathbf 1}_{#1}}  
\newcommand{\fun}[1]{\mathfrak{f}_{#1}} 
\newcommand{\intsf}{\mathsf{Int}} 
\newcommand{\abf}[2]{\langle #1, #2\rangle} 
\newcommand{\rbf}[2]{\left( #1 | #2 \right)} 
\newcommand{\rls}[1]{\scrR} 
\newcommand{\mrls}[1]{\scrR_{\scsop{min}}} 
\newcommand{\crls}[1]{\ol{\scrR}} 

\newcommand{\xg}[2]{x^{#1}_{#2}}

\newcommand{\xz}[1]{\xi_{#1}}
\newcommand{\xp}[1]{\xg{+}{#1}}
\newcommand{\xm}[1]{\xg{-}{#1}}
\newcommand{\xpm}[1]{\xg{\pm}{#1}}

\newcommand{\uhg}{U_\hbar(\g)}
\newcommand{\uhgx}{U_\hbar(\,\g_X)}
\newcommand{\uhbxp}{U_\hbar\big(\b^+_X\big)}
\newcommand{\uhbxm}{U_\hbar\big(\b^-_X\big)}
\newcommand{\uhbxpm}{U_\hbar\big(\b^\pm_X\big)}

\newcommand{\utildehgx}{\widetilde{U}_\hbar(\,\g_X)}
\newcommand{\utildehbxp}{\widetilde{U}_\hbar\big(\b^+_X\big)}
\newcommand{\utildehbxm}{\widetilde{U}_\hbar\big(\b^-_X\big)}

\newcommand{\utildezgx}{\widetilde{U}_0(\,\g_X)}

\newcommand{\uqgx}{\calU_q(\,\g_X)}
\newcommand{\uqbxp}{\calU_q\big(\b^+_X\big)}
\newcommand{\uqbxm}{\calU_q\big(\b^-_X\big)}
\newcommand{\uqbxpm}{\calU_q\big(\b^\pm_X\big)}

\newcommand{\uqnxp}{\calU_q\big(\hskip1pt\n^+_X\big)}
\newcommand{\uqnxm}{\calU_q\big(\hskip1pt\n^-_X\big)}
\newcommand{\uqfx}{\calU_q\big(\,\f_X\big)}

\newcommand{\utildeqgx}{\widetilde{\calU}_q(\,\g_X)}
\newcommand{\utildeqbxp}{\widetilde{\calU}_q\big(\b^+_X\big)}
\newcommand{\utildeqbxm}{\widetilde{\calU}_q\big(\b^-_X\big)}

\newcommand{\utildeunogx}{\widetilde{\calU}_1(\,\g_X)}

\newcommand{\barXi}{\overline{\varXi}}
\newcommand{\barX}{\overline{X}}
\newcommand{\barK}{\overline{K}}
\newcommand{\barH}{\overline{H}}

\newcommand{\dcs}{\triangleright\negthinspace\negthinspace\triangleleft}

\newcommand{\qxg}[2]{X^{#1}_{#2}}

\newcommand{\qxz}[2]{K_{#1}^{#2}}
\newcommand{\qxp}[1]{\qxg{+}{#1}}
\newcommand{\qxm}[1]{\qxg{-}{#1}}
\newcommand{\qxpm}[1]{\qxg{\pm}{#1}}

\newcommand{\iM}[2]{{#1}\triangledown{#2}}
\newcommand{\im}[2]{{#1}\negthinspace\vartriangle\negthinspace{#2}}

\newcommand{\hext}[1]{{#1}[[\hbar]]}

\newcommand{\ia}{\alpha}
\newcommand{\ib}{\beta}
\newcommand{\ic}{\gamma}

\newcommand{\ca}[2]{\mathsf{a}_{#1#2}}

\newcommand{\qcb}[3]{\mathsf{b}^{#3}_{#1#2}}
\newcommand{\qcc}[3]{\mathsf{c}^{#3}_{#1#2}}

\newcommand{\qcr}[3]{\mathsf{r}^{#3}_{#1#2}}
\newcommand{\qcs}[3]{\mathsf{s}^{#3}_{#1#2}}

\newcommand{\triend}{\parbox{2mm}{\hfill} \hfill\mbox{\hspace{0.2mm}}\hfill$\triangle$}
\newcommand{\ocend}{\parbox{2mm}{\hfill} \hfill\mbox{\hspace{0.2mm}}\hfill$\oslash$}

\newtheorem{theorem}{Theorem}[subsection]
\newtheorem{proposition}[theorem]{Proposition}
\newtheorem{lemma}[theorem]{Lemma}
\newtheorem{corollary}[theorem]{Corollary}

\newtheorem{def-thm}{Definition--Theorem}
\newtheorem{corollary*}{Corollary}
\newtheorem*{theorem*}{Theorem}
\newtheorem*{definition*}{Definition}

\newtheorem*{def-thm*}{Definition--Theorem}
\newtheorem*{proposition*}{Proposition}
\newtheorem*{conjecture*}{Conjecture}

\numberwithin{equation}{section}
\numberwithin{theorem}{subsection}

\theoremstyle{remark}
\newtheorem{ex}[theorem]{Example}

\theoremstyle{definition}
\newtheorem{Observation}[theorem]{Observation}
\newtheorem{rem}[theorem]{Remark}
\newtheorem{defin}[theorem]{Definition}
\newtheorem*{defin*}{Definition}
\newenvironment{definition}{\begin{defin}}{\ocend\end{defin}}

\title[Quantum Duality Principle for quantum continuous K--M algebras]
  {Quantum duality principle for  \\
   quantum continuous Kac--Moody algebras}

\author[Fabio GAVARINI]{Fabio GAVARINI}
\address{Fabio GAVARINI  ---  Dipartimento di Matematica,
Universit\`a degli Studi di Roma ``Tor Vergata''  --  I-00133 Roma, ITALY  \quad --- \quad   e-mail: {\tt gavarini@mat.uniroma2.it}}

\thanks{Partially supported by the MIUR  \textsl{Excellence Department Project\/}  awarded to the Department of Mathematics of the University of Rome ``Tor Vergata'', CUP E83C18000100006.}

\subjclass[2020]{Primary: 17B37, 20G42; Secondary: 17B65, 17B62}
\keywords{Continuous Kac-Moody algebras; continuous quantum groups; quantization of Lie bialgebras; quantization of Poisson groups.}

\begin{document}

{\ }

\vskip-51pt

   \centerline{\small  \textsl{Journal of Lie Theory\/}  \textbf{32}  (2022), no.\ 2, 839--862}
 \vskip5pt
   \centerline{\small --- preprint  {\tt arXiv:2202.06090 [math.QA]}  (2022) --- }
 \vskip4pt
   \centerline{\small  \textsl{The original publication is available at}}
 \vskip1pt
   \centerline{\small \texttt{https://www.heldermann.de/JLT/JLT32/JLT323/jlt32039.htm}}

\vskip27pt   {\ }

\begin{abstract}
  For the quantized universal enveloping algebra  $ \uhgx $  associated to a continuous Kac-Moody algebra  $ \g_X $  as in  \cite{appel-sala-20},  we prove that a suitable formulation of the  \textsl{Quantum Duality Principle\/}  holds true, both in a ``formal'' version   --- i.e., applying to the original definition of  $ \uhgx $  as a  \textsl{formal\/}  QUEA over  $ \kh $  ---   and in a ``polynomial'' one   --- i.e., for a suitable polynomial form of  $ \uhgx $  over  $ \kqqm \, $.  In both cases, the QDP states that a suitable subalgebra of the given quantization of the Lie bialgebra  $ \g_X $  is in fact a suitable quantization (in formal or in polynomial sense) of a connected Poisson group  $ G_X^* $  dual to  $ \g_X \, $.
\end{abstract}

\maketitle

\thispagestyle{empty}

\vskip1cm

\tableofcontents

\section{Introduction}  \label{introduction}
   Quantum groups, in their standard formulation as suitable topological Hopf algebras on a ring of formal power series  $ \kh \, $,  exist in two versions.  Namely, our quantum group is called a quantized universal enveloping algebra (or QUEA, in short), when its specialization at  $ \, \hbar = 0 \, $  is the universal enveloping algebra of some Lie algebra (actually a  \textsl{Lie bialgebra\/}),  or a quantum formal series Hopf algebra (in short, a QFSHA) when its specialization is (the algebra of functions on) a formal algebraic/Lie group   --- actually, a  \textsl{Poisson group}.  The categories of QUEA's and of QFSHA's are antiequivalent to each other via linear duality, just like it happens for their semiclassical counterparts.  Surprisingly enough, they are also  \textsl{equivalent},  through explicit equivalence functors, originally sketched in  \cite[\S 7]{drinfeld-quantum-groups-87},  and later detailed in  \cite{gavarini-02}:  in a sloppy formulation, this phenomenon is known as  \textsl{Quantum Duality Principle}   --- hereafter shortened as QDP.
                                                         \par
   Roughly speaking, the QDP claims that every QUEA, resp.\ every QFSHA, can be ``renormalized'' as to give rise to a QFSHA, resp.\ to a QUEA: in either case, the new quantum algebra   --- sometimes called ``the Drinfeld-Gavarini dual'' of the original one ---   is a quantization of the object (Poisson group or Lie bialgebra, respectively) which is  \textsl{Poisson dual\/}  to the object that the original quantum algebra is a quantization of.  In particular, if  $ \uhg $  is a QUEA quantizing  $ U(\g) $  then the QDP provides an explicit, functorial construction of a suitable Hopf subalgebra  $ {\uhg}' $  of  $ \uhg $  which is a quantization of  $ F[[G^*]] \, $,  where  $ G^* $  is the formal Poisson group dual to the Lie bialgebra  $ \g \, $.
 In fact, by construction  $ {\uhg}' $  is in fact a  $ \kh $--integral  form of  $ \, \Bbbk(\!(\hbar)\!) \otimes_\kh \uhg \, $,  \,just like  $ \uhg $  itself is.  In the other direction, if  $ F_\hbar[[G]] $  is any QFSHA for the formal Poisson group  $ G $  then the QDP provides a different  $ \kh $--integral  form  $ {F_\hbar[[G]]}^\vee $  of  $ \, \Bbbk(\!(\hbar)\!) \otimes_\kh F_\hbar[[G]] \, $  that is indeed a QUEA for  $ \g^* \, $.
                                                         \par
   Note that the geometrical objects  $ \g $  and  $ G $  (and their Poisson dual) considered by the QDP in its original formulation are finite dimensional, though some aspects of its functorial construction do apply to the infinite setup as well.
 \vskip5pt
   In a different approach, where quantum groups are defined as standard (i.e., non-topological) Hopf algebras over the field  $ \Bbbk(q) $   --- that is, \`a la Jimbo-Lusztig, say ---   so that one deals with ``polynomial QUEA'' and ``polynomial QFA (=Quantum Function Algebras)'', a suitable polynomial version of the QDP has been developed  (cf.\ \cite{gavarini-02}).  In short, in this context one considers a Hopf algebra  $ \mathbb{H} $ over  $ \Bbbk(q) $  and an  $ \kqqm $--integral  form  $ H $  of it: the latter then are called QUEA or QFA depending on whether  $ \; H \Big/ (\,q-1)\,H \; $  has the form  $ U(\g) $  or  $ F[G] \, $,  \,whence one writes  $ \, H = \calU_q(\g) \, $  or  $ \, H = \calF_q[G] \, $,  \,respectively.  Then the ``polynomial'' QDP provides functorial recipes (direct adaptation of Drinfeld's original ones)  $ \, \calU_q(\g) \mapsto {\calU_q(\g)}' \, $  and  $ \, \calF_q[G] \mapsto {\calF_q[G] }^\vee \, $  such that  $ {\calU_q(\g)}' $  is a QFA for  $ G^* $  and  $ {\calF_q[G] }^\vee $  is a QUEA for  $ \g^* $  (actually, the complete result is much stronger, see  \cite[Theorem 2.2]{gavarini-02}).  In particular,  $ \calU_q(\g) $  and  $ {\calU_q(\g)}' $  are two  $ \kqqm $--integral  forms of the same  $ \mathbb{H} $,  and similarly for  $ \calF_q[G] $  and  $ {\calF_q[G] }^\vee \, $.  Indeed, in concrete examples, when the  $ \Bbbk(q) $--algebra  $ \mathbb{H} $  is given by a presentation by generators and relations, the difference between the two integral forms  $ \calU_q(\g) $  and  $ {\calU_q(\g)}' $  amounts to a different choice of generators (roughly, a different ``normalization'' of them), and similarly for  $ \calF_q[G] $  and  $ {\calF_q[G] }^\vee $  again.  For instance, for the usual Jimbo-Lusztig quantum group  $ \mathbb{U}_q(\g) $  over a finite-dimensional semisimple  $ \g $  one can realize that (up to details)  $ \calU_q(\g) $  is nothing but Lusztig's  \textsl{restricted\/}  form, while  $ {\calU_q(\g)}' $ is De Concini-Procesi`s  \textsl{unrestricted\/}  one.  As both can be defined even over  $ \Z\big[\,q\,,q^{-1}\big] \, $,  \,thus leading to (different) theories of  \textit{quantum groups at roots of 1\/},  we also see how the polynomial QDP is somehow deeply intertwined with the theory of quantum groups at roots of 1   --- although a formal, sound and complete theory about that correlation has still to be unveiled.
 \vskip7pt
   To date, the impact of the QDP   --- either in formal or in polynomial version ---   on the development of quantum group theory has been paramount, in a pervasive manner (although not always explicitly recognized).  Nevertheless, as in real life examples and constructions of QUEA's are available way more than of QFSHA's (or QFA's), people mostly applied the QDP in the direction  $ \, \text{QUEA} \mapsto \text{QFSHA} \, $  (or  $ \, \text{QUEA} \mapsto \text{QFA} \, $,  in the polynomial case).
                                                          \par
   For instance, the formal QDP was used in the very construction of (formal) QUEA's of Lie bialgebras   --- cf.\ \cite{appel-toledanolaredo-18, appel-toledanolaredo-19, etingof-kazhdan-96, enriquez-01, enriquez-05, halbout-06, enriquez-halbout-07}  ---   sometimes even extending its range of application to the infinite-dimensional framework.  These results were also extended to broader contexts, such as that of quasi-Hopf QUEA (over quasi-Lie bialgebras) --- cf.\ \cite{enriquez-halbout-04}  ---   that of super Lie bialgebras   --- cf.\ \cite{geer-06}  ---   that of (quantum) groupoids   --- cf.\ \cite{chemla-gavarini-15}  ---   that of (quantization of)  $ \Gamma $--Lie  bialgebras and Poisson-Hopf stacks over groupoids   --- cf.\ \cite{enriquez-halbout-08, halbout-xiang_tang-10}  ---   and that of Yangians   --- cf.\ \cite{kamnitzer-webster-weekes-yacobi-10, finkelberg-tsymbaliuk-19}.  In another direction, the formal QDP was also applied to study quantum  $ R $--matrices  and associated structures (and variations on this theme), both from a geometrical point of view and a representation-theoretic one, as in  \cite{gavarini-halbout-01, gavarini-halbout-03, enriquez-gavarini-halbout-03, enriquez-etingof-marshall-05}.  Another geometrical application was to quantum homogeneous spaces, as in  \cite{ciccoli-gavarini-06},  where the QDP was suitably extended to formal quantizations (both infinitesimal and global) of Poisson homogeneous spaces.
                                                          \par
   On the other hand, the polynomial version of the QDP is applied to (or is definitely underlying) the construction and study of new QFA's   --- in a finite dimensional setup (cf.\ \cite{deconcini-procesi} for the uniparameter case, and \cite{gavarini_PJM-98,garcia-gavarini-??} for the multiparameter case)  or an infinite one  (cf.\ \cite{beck-94, beck-96, beck-kac, gavarini-00})  ---   or new QUEA's   --- in a finite (cf.\ \cite{gavarini_CA-98, gavarini-rakic})  or infinite  (cf.\ \cite{gavarini-00})  dimensional setup.  In a more geometrical perspective, it was applied   --- again in the ``direction''  $ \, \text{QUEA} \mapsto \text{QFA} \, $  to the study of quantum  $ R $--matrices  (and related subjects) in  \cite{gavarini-97, gavarini-01}   --- respectively for finite and affine type Kac--Moody Lie bialgebras, and in the study of Poisson homogeneous spaces in  \cite{ciccoli-fioresi-gavarini-08}   --- where a suitable version of polynomial QDP is tailored ad hoc for the projective case ---   in  \cite{fioresi-gavarini-11}   --- where quantum Grassmannians are treated ---   and in  \cite{ciccoli-gavarini-14}   --- where a more general construction is provided.  Still on a geometrical side, in the wake of a very fruitful research line, the polynomial QDP was applied in  \cite{habiro-le-16}  to provide a new topological invariant of integral homology spheres.
                                                          \par
   Finally, despite being a phenomenon that is intrinsically ``quantum'' in nature, the QDP (in polynomial version) had also found a remarkable application back in ``classical'' Hopf algebra theory   --- cf.\ \cite{gavarini_JA-05}  ---   with lot of immediate applications at hand   --- e.g., see  \cite{gavarini_CMP-05}  for a pretty interesting example.
 \vskip7pt
   The purpose of the present work is to prove yet another instance of the QDP, both formal and polynomial, namely in the direction  $ \, \text{QUEA} \mapsto \text{QF(SH)A} \, $  for the quantization of the  \textsl{continuous Kac--Moody algebras\/}  by Appel, Sala and Schiffmann  (see \cite{appel-sala-schiffmann-18, appel-sala-20}.  Indeed, these (topological) Lie bialgebras, hereafter denoted by  $ \g_X \, $,  are uncountably infinite-dimensional, hence one cannot directly apply the QDP as stated and proved in  \cite{gavarini-02}.  Instead, starting from the formal QUEA  $ \uhgx $  we provide a direct definition of a suitable subalgebra  $ \utildehgx $  of  $ \uhgx $  and then we prove that it has exactly the properties predicted by the QDP, in particular  $ \utildehgx $  is a QFSHA whose semiclassical limit is  $ F\big[\big[G^*\big]\big] \, $.  Finally, we also prove that this  $ \utildehgx $  actually admits also a description that coincides with the one prescribed by the usual Drinfeld's functor  $ \, \uhg \mapsto {\uhg}' \, $.
                                                                     \par
   As a second step, we introduce a suitable  \textsl{polynomial\/}  QUEA  $ \uqgx $   --- easy to guess as a subalgebra of  $ \uhgx $  ---   and we realize for it the (polynomial) QDP by introducing by hands its appropriate Drinfeld-Gavarini dual  $ {\uqgx}' \, $.  Here again, we cannot apply the general recipe given in  \cite{gavarini-07}   (as the latter applies to the finite dimensional case only), but we give instead a direct definition of a suitable integral form  $ \utildeqgx \, $,  \,inspired by what is done for  $ \calU_q(\g) $  when  $ \g $  finite Kac--Moody  (cf.\  \cite{deconcini-procesi}  and  \cite{gavarini_PJM-98})  or affine Kac--Moody (see  \cite{beck-96, beck-kac}  and  \cite{gavarini-00}).  Later on, we also prove that this  $ \utildeqgx $  does coincide with what comes out if one literally applies the recipe for Drinfeld's functor  $ \, \calU_q(\g) \mapsto {\calU_q(\g)}' \, $  as given in  \cite{gavarini-07}.
                                                                     \par
   An important feature of the construction sketched above is the following.  In the ``indirect'' construction, mentioned above, of the Drinfeld-Gavarini dual  $ {\calU_q(\g)}' $  as a suitable  $ \kqqm $--integral  form of  $ \, \Bbbk(q) \otimes \calU_q(\g) \, $  when  $ \g $  is Kac--Moody finite or affine (as in the works of De Concini-Procesi, Beck and the author), a critical step is the construction of suitable ``quantum root vectors'' for any root, that are  \textsl{not\/}  available from scratch.  However, the Lie bialgebras  $ \g_X $  have a Kac--Moody like presentation which includes, as generators, the (analogue of) ``root vectors'' for all possible ``roots''; even more, the same is true for the QUEA's  $ \uhgx $  and  $ \uqgx $  alike.  Therefore, the ``critical step'' mentioned before is already fixed from scratch, so that performing the same construction of the  $ \kqqm $--integral  form  $ \, \utildeqgx = {\uqgx}' \, $  of  $ \, \Bbbk(q) \otimes \uqgx \, $  as mentioned above becomes an easy task.  Up to technicalities, the very same strategy can be followed in order to define  $ \, {\uhgx}' = \utildehgx \, $,  \,again because all needed quantum root vectors are already given by definition.
                                                                     \par
   As a last remark, we point out that both the QUEA  $ \uqgx $   --- for the (topological) Lie bialgebra $ \g_X $  ---   and the QFA  $ \, \utildeqgx = {\uqgx}' \, $   --- for the Poisson group  $ G_X^* $  ---   are actually defined over  $ \Z\big[\,q\,,q^{-1}\big] \, $:  \,hence, an ``arithmetic theory'' for specializations at roots of 1, much like Lusztig (for the QUEA side) and De Concini-Procesi and Beck (for the QFA side) did, in principle is at hand.

\vskip11pt

   \centerline{\sc acknowledgements}
 \vskip5pt
   The author wishes to thank Andrea Appel and Margherita Paolini for several useful conversations.

\bigskip

\section{Preliminaries}  \label{sec: prel's}
 In this section, we briefly recollect from the literature the main material that we shall deal with.

\medskip

\subsection{Quantization of Lie bialgebras and of (formal) Poisson groups}  \label{subsec: quant_Lie-bialg_&_P-groups}
 Hereafter we fix a base field  $ \Bbbk $  of characteristic zero.  We recall the following from  \cite{chari-pressley}.
 \vskip5pt
   For any Lie algebra  $ \g $  over  $ \Bbbk \, $,  its universal enveloping algebra  $ U(\g) $  has a canonical structure of Hopf algebra, which is cocommutative and connected.  If  $ \g $  is also a Lie bialgebra, with Lie cobracket  $ \delta \, $,  then  $ \delta $  uniquely extends to define a  \textsl{Poisson cobracket\/}  $ \, \delta : U(\g) \longrightarrow U(\g) \otimes U(\g) \, $,  just by imposing that it fulfill the co-Leibnitz identity  $ \; \delta(x\,y) \, = \, \delta(x) \, \Delta(y) + \Delta(x) \, \delta(y) \; $.  Conversely, if the Hopf algebra  $ U(\g) $  is actually even a Hopf  \textsl{co-Poisson\/}  algebra, then its Poisson co-bracket  $ \delta $  maps  $ \g $  into  $ \, \g \otimes \g \, $,  thus yielding a Lie cobracket for  $ \g $  that makes the latter into a Lie bialgebra.
                                                                    \par
   Dually, let  $ G $  be any formal algebraic group  $ G $  over  $ \Bbbk \, $:  by this we loosely mean that  $ G $  is the spectrum of its formal function algebra  $ F[[G]] \, $,  the latter being a topological Hopf algebra which is commutative and  $ I $--adically  complete, where  $ \, I := \textsl{Ker}\,(\epsilon) \, $  is the augmentation ideal of $ F[[G]] \, $.  Then  $ G $  is a (formal) Poisson group if and only if its formal function algebra  $ F[[G]] $  is actually a  \textsl{Poisson\/}  (formal) Hopf algebra, with respect to some Poisson bracket  $ \{\,\ ,\ \} \, $.  In this case, the cotangent space  $ \, I \big/ I^2 \, $  of  $ G \, $,  has Lie bracket induced by  $ \{\,\ ,\ \} $  via  $ \; [x,y] := \big\{ x' , y' \big\} \; \big( \text{mod\ } I^2 \,\big) \; $  for all  $ \, x , y \in I \big/ I^2 \, $  with  $ \, x = x' \; \big( \text{mod\ } I^2 \,\big) \, $,  $ \, y = y' \; \big( \text{mod\ } I^2 \,\big) \, $:  this makes  $ I \big/ I^2 $ into a Lie algebra, but its dual  $ \, \g = \textsl{Lie}\,(G) := {\big( I \big/ I^2 \,\big)}^* \, $  is also a Lie algebra (the tangent Lie algebra to  $ G \, $)  and the two structures are compatible, so that  $ \, \g^\star := I \big/ I^2 \, $  is a  \textsl{Lie bialgebra\/}  indeed.
 \vskip5pt
   We come now to  \textsl{quantizations\/}  of the previous co-Poisson/Poisson structures.
 \vskip7pt
   \textsl{$ \underline{\text{QUEA}} $:}\,  A  \emph{quantized universal enveloping algebra\/}  (QUEA) is a (topological) Hopf algebra  $ U_\hbar $  in  $ \vect_\kh $  such that
\begin{enumerate}  \itemsep0.2cm
	\item   $ U_\hbar $  is topologically complete with respect to the  $\hbar$--adic  topology --- or, equivalently,  $ U_\hbar $ is isomorphic, as a topological  $ \Bbbk $--module,  to  $ \hext{U_0} \, $,  where  $ \; U_0 \cong U_\hbar \big/ \hbar\,U_\hbar \; $  is seen as a discrete topological vector space;
	\item   $ U_0 \:= U_\hbar \big/ \hbar\,U_\hbar \, $  is a connected, cocommutative Hopf algebra over  $ \Bbbk $   --- or, equivalently,  $ \, U_0 \, $  is isomorphic to an enveloping algebra  $ U(\g) $	for some Lie algebra  $ \g \, $;  \,then the formula
 $ \; \displaystyle{ \delta(x) \, := \, \frac{\; \Delta\big(x'\big) - \Delta^{21}\big(x'\big) \;}{\hbar} \; \mod \hbar\,U_\hbar^{\,\widehat{\otimes}\,2} } \; $
 --- where  $ \, x' \in U_\hbar \, $  is any lift of  $ \, x \in \g \, $  ---   defines a co-Poisson structure on  $ \, U_0 = U(\g) \, $,  \,hence a Lie bialgebra structure on  $ \g \, $.
\end{enumerate}
 In this case, we say that  $ U_\hbar $  is a  \emph{quantization\/}  of the co-Poisson Hopf algebra  $ U(\g) \, $,  or (with a slight abuse of language) of the Lie bialgebra  $ \g \, $.
 \vskip5pt
   \textsl{$ \underline{\text{QFSHA}} $:}\,  A  \emph{quantized formal series Hopf algebra\/}  (QFSHA) is a (topological) Hopf algebra  $ F_\hbar $  in  $ \vect_\kh $  such that
\begin{enumerate}  \itemsep0.2cm
	\item   $ F_\hbar $  is topologically complete with respect to the  $ I_\hbar$--adic  topology, where  $ \, I_\hbar := \textsl{Ker}\big(\epsilon_{F_\hbar}\big) + \hbar\,F_\hbar \, $;
	\item   $ F_0 := F_\hbar \big/ \hbar\,F_\hbar \, $  is a commutative,  $ I $--adically  complete topological Hopf algebra over  $ \Bbbk \, $,  where  $ I $  is the augmentation ideal   --- or, equivalently,  $ \, F_0 \, $  is isomorphic to the algebra of functions of formal algebraic group  $ F[[G]] \, $;  then the formula
 $ \; \displaystyle{ \{x,y\} \, := \, \frac{\; \big[ x' , y' \big] - \big[ y' , x' \big] \;}{\hbar} \; \mod \hbar\,F_\hbar } \; $
 --- where  $ \, x' , y' \in F_\hbar \, $  are lifts of  $ \, x, y \in F[[G]] \, $  ---   defines a Poisson bracket in  $ F[[G]] \, $,  thus making  $ G $  into a (formal) Poisson group.
\end{enumerate}
 In this case, we say that  $ F_\hbar $  is a  \emph{quantization\/}  of the Poisson Hopf algebra  $ F[[G]] \, $,  or (with a slight abuse of language) of the formal Poisson group  $ G \, $.

\medskip

\subsection{Continuous Kac-Moody Lie bialgebras}  \label{subsec: cont-KM-bialg's}
 Hereafter we shall shortly recall the notion of  \emph{continuous Kac-Moody Lie bialgebras},  following  \cite{appel-sala-20}  and references therein, where the reader may find the (many) details that we shall skip.
                                                         \par
   By a  \emph{vertex space}\/  $ X $  we mean, roughly, a Hausdorff topological space locally modeled over  $ \R \, $.  Then one lifts the notion of \emph{connected interval}\/  from  $ \R $  to  $ X \, $,  in such a way that the set of all possible \emph{intervals in  $ X $},  denoted  $ \intsf(X) \, $,  is naturally endowed with two  \emph{partially defined}\/  operations: a sum  $ \oplus $,  given by concatenation of intervals, and a difference $\ominus$, given by set difference whenever the outcome is again in  $ \intsf(X) \, $.  Moreover, generalizing a standard tool for quivers, the set  $ \intsf(X) $   is equipped with
%
%
 a non-symmetric bilinear form  $ \, \abf{\,\cdot\,}{\,\cdot\,} \colon \intsf(X) \times \intsf(X) \longrightarrow \Z \, $  along with its symmetrization  $ \, \rbf{\,\cdot\,}{\,\cdot\,} \colon \intsf(X) \times \intsf(X) \longrightarrow \Z \, $   --- the  \emph{Euler form}\/  on  $ \intsf(X) \, $.
%
%
                                                  \par
   We refer to the datum  $ \; \calQ_X \, := \, \big(\, \intsf(X), \oplus,\ominus,\abf{\,\cdot\,}{\,\cdot\,}, \rbf{\,\cdot\,}{\,\cdot\,} \big) \, $  as the  \emph{continuous quiver}\/  of  $ X $.  Hereafter, we denote by  $ \fun{X} $ the  $ \Z $--span   --- inside the space of functions  $ \Bbbk^X $  ---   of the characteristic functions  $ \cf{\ia} \, $,  $ \, \ia \in \intsf(X) \, $.  \textsl{Note\/}  that in  $ \fun{X} $  the relations  $ \; \cf{\ia \oplus \ib} = \delta_{\ia \oplus \ib} (\cf{\ia} + \cf{\ib}) \; $   --- for all  $ \, \alpha , \beta \in \intsf(X) \, $  ---   hold true, where  $ \, \delta_{\ia \oplus \ib} := 1 \, $  if the sum  $ \, \ia \oplus \ib \, $  is actually defined and  $ \, \delta_{\ia \oplus \ib} := 0 \, $  otherwise.
                                                  \par
   In  \cite{appel-sala-schiffmann-18}  (where  $ \, \Bbbk := \C \, $),  for every continuous quiver  $ \calQ_X $  a Lie algebra  $ \g_X $  is constructed, called the \emph{continuous Kac-Moody algebra of  $ \, \calQ_X $},  whose definition mimics that of Kac-Moody algebras.  Namely, we first consider the Lie  $ \Bbbk $--algebra  $ \wt{\g}_X $  generated by  $ \fun{X} $  and the elements $\xpm{\ia}$, $\ia\in\intsf(X)$,  with relations
  $$  \big[ \xz{\ia} , \xz{\ib} \big] \, = \, 0 \;\; ,  \quad  \big[ \xz{\ia} , \xpm{\ib} \big] \, = \, \pm\rbf{\ia}{\ib} \xpm{\ib} \;\; ,  \quad  \big[ \xp{\ia} , \xm{\ib} \big] \, = \, \drc{\ia\ib}\xz{\ia} + \ca{\ia}{\ib} \big( \xp{\ia\ominus \ib} - \xm{\ib \ominus \ia} \big)  $$
 where  $ \, \xz{\ia} := \cf{\ia} \, $  and  $ \, \ca{\ia}{\ib} := {(-1)}^{\abf{{\ia} }{{\ib} }} \rbf{\ia}{\ib} \, $.  \textsl{Note\/}  that this implies that also all the relations
  $$  \xi_{\alpha \oplus \beta}  \; = \;  \delta_{\alpha \oplus \beta} (\xi_\alpha + \xi_\beta)   \eqno  \forall \;\; \alpha , \beta \in \intsf(X)   \qquad  $$
 hold true, just because  $ \; \cf{\ia \oplus \ib} = \delta_{\ia \oplus \ib} (\,\cf{\ia} + \cf{\ib}) \; $  in  $ \fun{X} \, $.  Then we endow  $ \wt{\g}_X $  with a suitable grading and we set  $ \, \g_X := \wt{\g}_X \big/ \r_X \, $,  \,where  $ \, \r_X \subseteq \wt{\g}_X \, $  is the sum of all two--sided graded ideals having trivial intersection with  $ \fun{X} \, $.  An explicit description of  $ \r_X $  is given in \cite{appel-sala-schiffmann-18},  thus eventually one explicitly describes  $ \g_X $  as the Lie  $ \Bbbk $--algebra  generated by the elements  $ \, \xz{\ia} , \xpm{\ia} \, $  ($ \, \ia \in \intsf(X) \, $),  \,subject to the following relations:
  $$  \displaylines{
   \hfill   \big[ \xz{\ia} , \xz{\ib} \big] \, = \, 0  \quad ,  \qquad
 \xi_{\alpha \oplus \beta}  \; = \;  \delta_{\alpha \oplus \beta} (\xi_\alpha + \xi_\beta)   \hfill \quad \qquad  \forall \;\; \ia , \ib \in \intsf{X}   \quad  \cr
   \hfill   \big[ \xz{\ia} , \xpm{\ib} \big] = \pm \rbf{\ia} {\ib} \xpm{\ib}  \; ,  \quad
   \big[ \xp{\ia} , \xm{\ib} \big] = \drc{\ia\ib} \, \xz{\ia} + \ca{\ia}{\ib} \left( \xp{\ia\ominus\ib} - \xm{\ib\ominus\ia} \right)   \hfill \quad  \forall \,\; \ia , \ib \in \intsf(X)  \cr
   \hfill   \big[ \xpm{\ia} , \xpm{\ib} \big] \, = \, \pm\,\ca{\ia}{,\, \ia\oplus\ib} \cdot \xpm{\ia\oplus\ib}   \hfill \qquad  \forall \;\; (\ia,\ib) \in \serre{X}   \quad  }  $$
 where  $ \serre{X} $  is the set of all pairs  $ \, (\ia,\ib) \in \intsf(X) \times \intsf(X) \, $  that obey some suitable conditions.  The general formulation of these conditions is quite technical, given in terms of the sum among elements of  $ \intsf(X) $,  of the Euler form  $ \, \rbf{\,\cdot\,}{\,\cdot\,} \, $, and a notion of ``partition'' of an interval  $ \, x_\ia \in \intsf(X) \, $  as (roughly speaking) an ordered decomposition of it into sum of ``smaller'' intervals.  Albeit looking quite tricky, this notion in fact has a very natural motivation; nevertheless, the technicalities reach far beyond our present scope, so we skip them (we shall not really them, indeed) referring instead to the original source  \cite{appel-sala-schiffmann-18}  for the interested reader.
                                                  \par
   Note that the  $ \Z $--span  of the  $ \xz{\ia} $'s  is a copy of  $ \fun{X} $  inside  $ \g_X \, $.
 \vskip3pt
   In addition, naturally defined  \emph{Borel subalgebras}\/  $ \b_X^+ $  and  $ \b_X^- $  exist in  $ \g_X \, $,  \,namely  $ \b_X^\pm $  is the Lie subalgebra of  $ \g_X $  generated by all the  $ \xi_\ia $'s  and the  $ \xpm{\ib} $'s  with  $ \, \ia , \ib \in \intsf(X) \, $.
 \vskip3pt
   Finally, it is also shown that  $ \g_X $  bears a (canonical) structure of quasitriangular,  \textsl{topological\/}  Lie bialgebra, whoses Lie cobracket is given on generators by
  $$  \delta(\xz{\ia}) \, := \, 0  \quad ,   \qquad  \delta(\xpm{\ia}) \; := \; \xz{\ia} \wedge \xpm{\ia} + {\textstyle \sum_{\ib \oplus \ic = \ia}} \, \ca{\ib}{,\,\ia} \, \xpm{\ib} \wedge\xpm{\ic}  $$
   --- the sum being convergent, in a natural sense.  Then both Borel subalgebras  $ \b_X^{\pm} $  are Lie sub-bialgebras, and the Euler form restricts to a non-degenerate pairing of Lie bialgebras  $ \; \rbf{\,\cdot\,}{\,\cdot\,} \colon \b_X^+ \otimes {\big( \b_X^- \big)}^{\scsop{cop}} \!\relbar\joinrel\longrightarrow \Bbbk \; $.  It follows that the canonical element  $ \, r_X \in \b_X^+ \wh{\ten} \b_X^- \, $  corresponding to  $ \rbf{\,\cdot\,}{\,\cdot\,} $  defines a  \emph{quasi-triangular structure\/}  on  $ \g_X \, $.

\medskip

\subsection{Quantization of continuous Kac-Moody Lie bialgebras}  \label{subsec: quantiz_cont-KM-bialg's}
 We now recall the construction of QUEA's that provide quantizations of the continuous Kac-Moody Lie bialgebras mentioned above: we follow again  \cite{appel-sala-20},  still referring to that source for all the details that we shall skip.
                                                         \par
   Given a continuous quiver  $ \calQ_X $  and the associated continuous Kac-Moody Lie bialgebra  $ \g_X \, $,  \,one defines also a suitable QUEA  $ \uhgx \, $:  its definition is modeled on that of  $ \g_X \, $,  but it depends on two additional partial operations on  $ \intsf(X) \, $:  the
 \emph{strict union of two non-orthogonal intervals  $ \ia $  and  $ \ib $},
 \,denoted  $ \, \iM{\ia}{\ib} \, $,  \,and the
 \emph{strict intersection of two non-orthogonal intervals  $ \ia $  and  $ \ib $},
 \,denoted  $ \, \im{\ia}{\ib} \, $  (roughly, these are slight variations of set-theoretical union and intersection, respectively).
                                                         \par
   Now, for given  $ \calQ_X $  and  $ \g_X $  as above, the  \emph{continuous quantum group of  $ X $}
is the unital,  $ \hbar $--adically  complete, associative  $ \kh $--algebra  $ \uhgx $  generated by $ \fun{X} $   --- whose spanning elements are now denoted by  $ \, \varXi_\alpha := \cf{\ia} \, $  ($ \, \alpha \in \intsf(X) \, $)  ---   and the elements  $ \qxpm{\ia} $  ($ \, \ia \in \intsf(X) \, )$,  \,obeying the following defining relations
 (for all  $ \, \ia , \ib \in \intsf(X) \, $,  with the additional constraint  $ \, (\ia,\ib) \in \serre{X} \, $  for the last relation)
%
  $$  \displaylines{
   \hfill   \varXi_\alpha \, \varXi_\beta - \varXi_\beta \, \varXi_\alpha = \, 0  \;\; ,  \;\;\;\;
 \varXi_{\alpha \oplus \beta}  = \,  \delta_{\alpha \oplus \beta} \big( \varXi_\alpha + \varXi_\beta \big)  \;\; ,  \;\;\;\;
 \varXi_\alpha \, \qxpm{\ib} - \qxpm{\ib} \, \varXi_\alpha \, = \, \pm\rbf{\ia}{\ib}\qxpm{\ib}   \hfill
%
%
  \cr
%
   \qxp{\ia} \, \qxm{\ib} - \qxm{\ib} \, \qxp{\ia}  \;\; = \;\;  \drc{\ia \ib} \frac{\qxz{\ia}{} - \qxz{\ia}{-1}}{q-q^{-1}}  \; +   \hfill  \cr
   \hfill   + \;  \ca{\ia}{\ib} \! \left( q^{\qcc{\ia}{\ib}{+}} \qxp{\ia\ominus\ib} \qxz{\ib}{\ca{\ia}{\ib}} \! - q^{\qcc{\ia}{\ib}{-}} \qxz{\ia}{\ca{\ia}{\ib}} \qxm{\ib\ominus\ia} \right)  \! +  \qcb{\ib}{\ia}{} q^{\qcb{\ib}{\ia}{}} \big( q - q^{-1} \big) \qxp{(\iM{\ia}{\ib})\ominus\ib} \qxz{\im{\ia\, }{\, \ib}}{\qcb{\ia}{\ib}{}} \, \qxm{(\iM{\ia}{\ib})\ominus{\ia}}
%
%
  \cr
%
%
   \hfill   \qxpm{\ia} \qxpm{\ib} \, - \,  q^{\qcr{\ia}{\ib}{}} \, \qxpm{\ib} \qxpm{\ia}  \; = \; \pm \qcb{\ia}{\ib}{} \, q^{\qcs{\ia}{\ib}{\pm}} \, \qxpm{\ia\oplus\ib} \, + \, \qcb{\ia}{\ib}{} \, \big( q - q^{-1} \big) \, \qxpm{\iM{\ia}{\ib}} \qxpm{\im{\ia\, }{\,\ib}}   \hfill
%
%
  }  $$
 where we set  $ \, q := \exp\big(\hbar/2\big) \, $  and  $ \, \qxz{\ia}{} := \exp\big( \hbar \, \varXi_\alpha /2 \big) \, $,  \,and we assume  $ \, \qxpm{\ia \odot \ib } := 0 \, $  whenever  $ \, \ia \odot \ib \, $  is not defined, for  $ \, \odot = \oplus , \ominus , \iM{}{}, \im{}{} \, $;  moreover, the various coefficients are defined as follows:
  $$  \displaylines{
   \ca{\ia}{\ib} \, := \, {(-1)}^{\abf{\ia}{\ib}}\rbf{\ia}{\ib}  \;\; ,  \;\quad
   \qcc{\ia}{\ib}{+} \, := \, 2^{-1} \left(\,\ca{\ib}{,\,\ia\ominus\ib}-1\right)  \;\; ,  \;\quad
   \qcc{\ia}{\ib}{-} \, := \, 2^{-1} \left(\,\ca{\ib\ominus\ia}{,\,\ia}+1\right)  \cr
   \qcb{\ia}{\ib}{} \, := \, \ca{\ia}{,\, \iM{\ia}{\ib}}  \;\; ,  \;\quad
   \qcr{\ia}{\ib}{} \, := \, (\,1-\drc{\ia\ib}) {(-1)}^{\abf{\ia}{\ib}}\rbf{\ia}{\ib}^2  \;\; ,  \;\quad
   \qcs{\ia}{\ib}{\pm} \, := \, 2^{-1} \left(\,\ca{\ib}{,\,\ia\oplus\ib}\pm1\right)  }  $$
   \indent   In addition, one also considers the  $ \kh $--subalgebras  $ \uhbxpm $  in  $ \uhgx $  generated by  $ \fun{X} $  and the  $ \qxpm{\ia} $'s  ($ \, \ia \in \intsf(X) \, $).  Moreover, both have an obvious presentation by generators and relations similar to that of  $ \uhgx $  (roughly speaking, one just has to pick only the relevant generators and the relations involving them).
 \vskip9pt
   The result concerning the Hopf algebra structure of  $ U_\hbar(\g_X) $  is the following:

\vskip11pt

\begin{theorem}[{cf.\ \cite[Theorem 5.11]{appel-sala-20}}]  \label{thm: Hopf-str_Uhgx}
 Let  $ \calQ_X $  be a continuous quiver and  $ U_\hbar(\g_X) $  the corresponding continuous quantum group.
 \vskip3pt
   (1)\,  The algebra  $ U_\hbar(\g_X) $  is a topological Hopf algebra with respect to the antipode and coproduct defined on the generators by
  $$  \displaylines{
   \epsilon\big(\varXi_\alpha\big) \, := \, 0 \, =: \, \epsilon\big(X_\alpha^\pm\big)  \quad ,   \qquad \qquad
\Delta(\xz{\ia})\ := \varXi_\alpha \ten 1 + 1 \ten \varXi_\alpha  \cr
   \Delta\big(\qxp{\ia}\big)  \, := \,  \qxp{\ia} \otimes 1 + K_\ia^{+1} \otimes \qxp{\ia} \, + \, {\textstyle \sum_{\ia = \ib \oplus \ic}} \, \ca{\ib}{, \,\ia} \, \qcs{\ic}{\ib}{+} \cdot \big(\, q - q^{-1} \big) \, K_\ic^{+1} \qxp{\ib} \otimes \qxp{\ic}  \cr
   \Delta\big(\qxm{\ia}\big)  \, := \,  1 \otimes \qxm{\ia} + \qxm{\ia} \otimes K_\ia^{-1} \, - \, {\textstyle \sum_{\ia = \ib \oplus \ic}} \, \ca{\ib}{, \,\ia} \, \qcs{\ic}{\ib}{-} \cdot \big(\, q - q^{-1} \big) \, \qxm{\ib} \otimes \qxm{\ic} K_\ib^{-1}  }  $$
%
%
 --- where  $ \Delta $  takes values in the completion of the algebraic tensor product with respect to the ``weak \&  $ \hbar $--adic  topology'' ---   while the antipode is given, as usual, by the formula
 $ \; S := \sum_n m^{(n)} \circ {(\id - \iota \circ \epsilon)}^{\otimes n} \circ \Delta^{(n)} \; $
 where  $ m^{(n)} $  and  $ \Delta^{(n)} $  denote the  $ n $--th  iterated product and coproduct.
 \vskip3pt
   (2)\,  $ \uhbxp \, $,  resp.\  $ \uhbxm \, $,  \,is a Hopf  $ \kh $--subalgebra  of  $ \, \uhgx \, $.
 \vskip3pt
   (3)\,  There exists a unique non-degenerate Hopf pairing
  $$  \rbf{\,\cdot\,}{\,\cdot\,} \colon \uhbxp \otimes {\big( \uhbxm \big)}^{\scsop{cop}} \!\relbar\joinrel\relbar\joinrel\longrightarrow \Bbbk((\hbar))  $$
 defined on the generators by the formulas
%
%
%
%
  $$  \rbf{1}{1} \, := \, 1  \quad ,  \qquad  \rbf{\,\varXi_\alpha\,}{\,\varXi_\beta} \, := \frac{\; \rbf{\ia}{\ib} \;}{\;\hbar\;}  \quad ,  \qquad  \big(\, \qxp{\ia} \,\big|\, \qxp{\ib} \,\big) \, := \frac{\delta_{\ia\ib}}{\; q - q^{-1} \;}  $$
 and zero otherwise.  In particular, one has  $ \, \big(\, \qxz{\ia} \;\big|\, \qxz{\ib} \,\big) = q^{\rbf{\ia}{\ib}} \, $.
%
 \vskip3pt
   (4)\,  Through the Hopf pairing  $ \rbf{\,\cdot\,}{\,\cdot\,} \, $,  the Hopf algebras  $ \, \uhbxp \, $  and  $ \, \uhbxm \, $  give rise to a match pair of Hopf algebras; via this,  $ U_\hbar(\g_X) $  is realized as a quotient of the double cross product Hopf algebra  $ \; \uhbxp \dcs\, \uhbxm \, $  obtained by identifying the two copies of the commutative subalgebra  $ \fun{X} \, $.  In particular,  $ U_\hbar(\g_X) $  is a topological quasi-triangular Hopf algebra.
 \vskip3pt
   (5)\,  The topological quasi-triangular Hopf algebra  $ \uhgx $  is a quantization of the topological quasi-triangular Lie bialgebra  $ \, \g_X \, $,  \,through the (unique) isomorphism  $ \; \uhgx \Big/ \hbar\,\uhgx \,\cong \, U(\g_X) \; $
%
%
 given    $ \; \varXi_\ia \,\mapsto\, \xi_\ia \; $  and  $ \; X_\ia^\pm \,\mapsto\, x_\ia^\pm \; $  for all  $ \, \ia \in \intsf(X) \, $.
                                                           \par
   Similarly,  $ \uhbxpm $  is a quantization of the Lie bialgebra  $ \b_X^\pm \, $.
\end{theorem}

\vskip9pt

\begin{rem}
 The original formulas in  \cite{appel-sala-20}  for  $ \Delta\big(X^\pm\big) $  were
  $$  \displaylines{
   \Delta\big(\qxp{\ia}\big)  \, := \,  \qxp{\ia} \otimes 1 + K_\ia \otimes \qxp{\ia} \, + \, {\textstyle \sum_{\ia = \ib \oplus \ic}} \, \ca{\ic}{, \,\ia} \, \qcs{\ib}{\ic}{-} \cdot q^{-1} \big(\, q - q^{-1} \big) \, \qxp{\ib} K_\ic \otimes \qxp{\ic}  \cr
   \Delta\big(\qxm{\ia}\big)  \, := \,  1 \otimes \qxm{\ia} + \qxm{\ia} \otimes K_\ia^{-1} \, - \, {\textstyle \sum_{\ia = \ib \oplus \ic}} \, \ca{\ic}{, \,\ia} \, \qcs{\ib}{\ic}{-} \cdot \big(\, q - q^{-1} \big) \, \qxm{\ib} \otimes \qxm{\ic} K_\ic^{-1}  }  $$
 where  \textsl{there is a misprint in second line\/}  (cf.\ \cite{appel-21}),  namely  $ \, \qxm{\ib} \otimes \qxm{\ic} K_\ic^{-1} \, $  should read instead  $ \, \qxm{\ic} \otimes \qxm{\ib} K_\ic^{-1} \, $.  Further re-writings are possible, giving for instance
  $$  \displaylines{
   \Delta\big(\qxp{\ia}\big)  \, := \,  \qxp{\ia} \otimes 1 + K_\ia^{+1} \otimes \qxp{\ia} \, + \, {\textstyle \sum_{\ia = \ib \oplus \ic}} \, \ca{\ib}{, \,\ia} \, \qcs{\ic}{\ib}{+} \cdot \big(\, q - q^{-1} \big) \, K_\ic^{+1} \qxp{\ib} \otimes \qxp{\ic}  \cr
   \Delta\big(\qxm{\ia}\big)  \, := \,  1 \otimes \qxm{\ia} + \qxm{\ia} \otimes K_\ia^{-1} \, + \, {\textstyle \sum_{\ia = \ib \oplus \ic}} \, \ca{\ib}{, \,\ia} \, \qcs{\ib}{\ic}{-} \cdot \big(\, q - q^{-1} \big) \, \qxm{\ic} \otimes \qxm{\ib} K_\ic^{-1}  \cr
   \text{or}   \hfill  \cr
   \Delta\big(\qxp{\ia}\big)  \, := \,  \qxp{\ia} \otimes 1 + K_\ia^{+1} \otimes \qxp{\ia} \, + \, {\textstyle \sum_{\ia = \ib \oplus \ic}} \, \ca{\ib}{, \,\ia} \, \qcs{\ic}{\ib}{+} \cdot \big(\, q - q^{-1} \big) \, K_\ic^{+1} \qxp{\ib} \otimes \qxp{\ic}  \cr
   \Delta\big(\qxm{\ia}\big)  \, := \,  1 \otimes \qxm{\ia} + \qxm{\ia} \otimes K_\ia^{-1} \, + \, {\textstyle \sum_{\ia = \ib \oplus \ic}} \, \ca{\ic}{, \,\ia} \, \qcs{\ic}{\ib}{-} \cdot \big(\, q - q^{-1} \big) \, \qxm{\ib} \otimes \qxm{\ic} K_\ib^{-1}  }  $$
\end{rem}

\bigskip

\section{QDP for formal quantum continuous KM-algebras}  \label{sec: QDP-formal}
 We are now ready to present the Quantum Duality Principle (=QDP) for the QUEA associated with  $ \g_X \, $,  namely  $ \uhgx \, $.  Indeed, we cannot get it as a direct application of the general result in  \cite{gavarini-02},  because  $ \g_X $  is infinite-dimensional.  However, from the general result in  [\emph{loc.\ cit.}]  we can still cook up a suitable formulation of the QDP expressly tailored as to match the case of  $ \uhgx \, $.
                                                             \par
   We begin here with a formulation of the QDP in the ``formal'' version   --- i.e., for quantizations ``\,\`a la Drinfeld'' such as the formal  $ \uhgx $  introduced in  \S \ref{subsec: quantiz_cont-KM-bialg's}  above.
 However, as our construction is somewhat indirect, and as such it might seem to come out of the blue, we first start by explaining the very motivation that guides our construction, shedding light on some key technical steps in the general setup.

\medskip

\subsection{Drinfeld's functor  $ \, \uhg \mapsto {\uhg}' \, $  and its description}
 We deal now with a QUEA  $ \uhg $  over a Lie  $ \Bbbk $--bialgebra  $ \g $   --- cf.\  \S \ref{subsec: quant_Lie-bialg_&_P-groups}.  The QDP as originally devised by Drinfeld applies to  $ \g $  of  \textsl{finite dimension\/}:  indeed, it is proved in  \cite{gavarini-02}  that most steps leading to the main result (the proof of the QDP in its full extent) actually still hold true in greater generality as well, but the whole result as such does not.  However, one can carefully modifies some assumptions   --- e.g., the nature of the (quantum) Hopf algebras we are dealing with, in particular their topology ---   so to finally achieve a suitably modified version of the QDP and some related byproducts.  Hereafter, we revisit in detail some steps of the analysis carried on in \cite{gavarini-02},  so that we will be later able to adapt them to the QUEA's  $ \uhgx $  on the infinite-dimensional (topological) Lie bialgebras  $ \g_X \, $.

\medskip

\begin{definition}  \label{def: Drinfeld's functors}
 \textit{(Drinfeld's functors)}  We define Drinfeld's functors on QUEA's and QFSHA's (only on objects --- we do not need them on morphisms) as follows:
 \vskip3pt
   \textit{(a)}\,  Let  $ \uhg $  be any QUEA  (cf.\  \S \ref{subsec: quant_Lie-bialg_&_P-groups}),
%
%
 and assume for simplicity that  $ \g $  be finite-dimensional.  Let  $ \, \iota : \kh \relbar\joinrel\longrightarrow \uhg \, $  and  $ \, \epsilon : \uhg \relbar\joinrel\longrightarrow \kh \, $  be its unit and counit maps; moreover, for every  $ \, n \in \N \, $  set  $ \, \delta_n := (\id - \iota \circ \epsilon) \circ \Delta^{(n)} \, $.  Then we define
 \vskip5pt
   \hfill   $ {\uhg}'  \; := \;  \Big\{\, \eta \in \uhg \,\Big|\; \delta_n(\eta) \in \hbar^n\,{\uhg}^{\otimes n} \;\; \forall \; n \in \N \,\Big\} $   \hfill {\ }
 \vskip3pt
   \textit{(b)}\,  Let  $ F_\hbar[[G]] $  be any QFSHA  (cf.\  \S \ref{subsec: quant_Lie-bialg_&_P-groups}),
%
%
 and assume for simplicity that  $ G $  be finite-dimensional.  Let  $ \, \epsilon_{\scriptscriptstyle F} : F_\hbar[[G]] \relbar\joinrel\longrightarrow \kh \, $  be its counit map, and consider also  $ \, I_{F_\hbar} := \hbar\,F_\hbar[[G]] + \textsl{Ker\/}(\epsilon_{\scriptscriptstyle F}) \, $.  Then we define
 \vskip5pt
   \qquad \qquad \hfill   $ {F_\hbar[[G]]}^\vee  \; := \;  \hbar $--adic  completion of  $ \, \sum_{n \geq 0} \hbar^{-n} I_{F_\hbar}^{\;{}^{\scriptstyle n}} $   \hfill
\end{definition}

\vskip11pt

   When  $ \g $  is finite-dimensional, one proves  (cf.\  \cite{gavarini-02})  that  $ {\uhg}' $  is a QFSHA quantizing  $ G^* \, $,  \,the formal Poisson group dual to  $ \g \, $  (see also  Remark \ref{rem: inf-dim case}  below for possible generalizations).  Moreover, the following, alternative description is possible (it was mentioned in \cite[\S 3.5]{gavarini-02},  but in a somewhat confused manner):

\vskip13pt

\begin{proposition}  \label{prop: exist-lifts-x_i}
 Given a\/  $ \Bbbk $--basis  $ {\{ \overline{y}_i \}}_{i \in I} $  of\/  $ \g \, $,  \,there exist  $ \, y_i \in \uhg \, $  such that:
 \vskip3pt
   (a)\;  $ \, \epsilon(y_i) = 0 \, $,  $ \, \big(\, y_i \!\mod \hbar\,\uhg \big) = \overline{y}_i \, $  and  $ \; y'_i := \hbar\,y_i \, \in \, {\uhg}' \, $  for all  $ \, i \in I \, $;
 \vskip3pt
   (b)\;  $ \, {\uhg}' $  is the completion of the unital\/  $ \kh $--subalgebra  of  $ \, \uhg $  generated by all the  $ x'_i $'s  with respect to its  $ I'_\hbar $--adic  topology, where  $ I'_\hbar $  is the ideal (in that subalgebra) generated by  $ \hbar $  and all the  $ x'_i $'s,  so that
  $ \; {\uhg}' = \Bbbk\big[\big[ {\{ x'_i \}}_{i \in I} \cup \{\hbar\} \big]\big] \; $.
%

%
\end{proposition}

\begin{proof}
 By  \cite[Proposition 3.6]{gavarini-02},  $ \, \calF_\hbar := {\uhg}' \, $  is a QFSHA.  Moreover, from  \cite[\S 3.1]{gavarini-02}  we have  $ \, \calF_\hbar = \Bbbk\big[\big[ {\{ x_i \}}_{i \in \calI} \cup \{\hbar\} \big]\big] \, $  for some  $ \, x_i \in \textsl{Ker\/}\big(\,\epsilon_{\calF_\hbar}\big) \, $   --- for each  $ \, i \in \calI \, $,   the latter being an index set ---   such that their cosets  $ \, \big(\, (\, x_i \!\mod \hbar\,\calF_\hbar \,) \!\mod J_\calF^{\,2} \,\big) \, $  form a  $ \Bbbk $--basis  of  $ \; J_\calF \big/ J_\calF^{\,2} \, \cong \, \g \, $,  \,where  $ J_\calF $  is the kernel of the counit of  $ \, \calF_\hbar \Big/ \hbar\,\calF_\hbar \; $.  In addition, by  \cite[Proposition 3.2]{gavarini-02}  we have that  $ \calF_\hbar^{\,\vee} $  is a QUEA, and we have an explicit description of it: namely, as a  $ \kh $--module  it is  $ \; \calF_\hbar^{\,\vee} = \big(\, \Bbbk\big[ {\{ \check{x}_i \}}_{i \in \calI} \big] \big)[[\hbar]] \, $,  \,i.e.\ it is the  $ \hbar $--adic  completion of the  $ \kh $--subalgebra  of  $ \, \Bbbk(\!(\hbar)\!) \otimes_\kh \calF_\hbar \, $  generated by the elements  $ \, \check{x}_i := \hbar^{-1} x_i \, $,  $ \, i \in \calI \, $.  Finally, its semiclassical limit is  $ \; \calF_\hbar^{\,\vee} \!\Big/\, \hbar\,\calF_\hbar^{\,\vee} \, = \, U(\mathfrak{k}^*) \; $,  \,with the cosets $ \, \check{x}_i \, \big(\, \text{mod} \; \hbar\,\calF_\hbar^{\,\vee} \,\big) \, $   ---  $ \, i \in \calI \, $  ---   forming a  $ \Bbbk $--basis  of  $ \mathfrak{k}^* \, $,  where  $ \mathfrak{k} $  is the tangent Lie bialgebra to the formal Poisson group $ K $  given by the semiclassical limit of  $ \calF_\hbar \, $,  i.e.\ such that  $ \, \calF_\hbar \Big/ \hbar\,\calF_\hbar= F[[K]] \, $.
 \vskip3pt
  Now, from  \cite[Proposition 2.2]{gavarini-02}  we have  $ \, {\big( {\uhg}' \,\big)}^\vee = \uhg \, $,  \,hence in particular  $ \, \calF_\hbar^{\,\vee} = \uhg \, $;  \,then the above gives  $ \; {\uhg}' = \calF_\hbar = \Bbbk\big[\big[ {\{ x_i \}}_{i \in \calI} \cup \{\hbar\} \big]\big] \; $  and  $ \, \mathfrak{k}^* = \g \, $,  \,whence we get the claim with  $ \, \overline{y}_i = \check{x}_i \, $  and  $ \, y_i = x_i \, $  for all  $ \, i \in I \, $.
\end{proof}

\vskip9pt

\begin{rem}  \label{rem: generation-ughx'-by-x'_i}
 It is worth stressing here a key point.  Proposition \ref{prop: exist-lifts-x_i}  above ensures that  $ {\uhg}' $  can be described as an algebra of ``formal series'' in the elements  $ \, y'_i = \hbar\,y_i \, $, which are re-scaling of suitable lifts in  $ \textsl{Ker\/}\big(\uhg\big) $  of elements  $ \overline{y}_i $  in a  $ \Bbbk $--basis  of  $ \g \, \big( \subseteq U(\g) \big) \, $.  However, we must point out that  \textsl{this heavily depends on the choice of these lifts},  in particular it is definitely  \textsl{false\/}  that the same might hold true with  \textsl{any\/}  lifts of the  $ \overline{y}_i $'s   --- easy counterexamples exist, already for  $ \, \g = \mathfrak{sl}_2 \, $  and its standard Drinfeld's quantization.
                                                            \par
   On the other, choosing inside  $ \, \textsl{Ker\/}\big(\uhg\big) \cap {\uhg}' \, $  any lift  $ {\{\,y_i\,\}}_{i \in \calI} $  of a  $ \Bbbk $--basis  of  $ \g $  as in claim  \textit{(a)\/}  of  Proposition \ref{prop: exist-lifts-x_i}  is indeed enough to guarantee that claim  \textit{(b)\/}  holds true as well; this is the content of next result, which we will apply later on.
\end{rem}

\vskip9pt

\begin{lemma}  \label{lemma: generation-ughx'-by-x'_i}
 Let  $ \, x'_i \in {\uhg}' \, $   --- for all  $ \, i \in I \, $  ---   have the following properties:  $ \, x'_i = \hbar\,x_i \, $  for some  $ \, x_i \in \uhg \, $  such that  $ \, \epsilon(x_i) = 0 \, $  (for all  $ \, i \in I \, $)  and the set  $ \, {\big\{\, \overline{x}_i := \big(\, x_i \!\!\mod \hbar\,\uhg \big) \big\}}_{i \in \calI} \, $  is a\/  $ \Bbbk $--basis  of\/  $ \g \, $.  Then  $ {\uhg}' $  is the  $ I''_\hbar $--adic  completion of the unital\/  $ \kh $--subalgebra  $ {\uhg}'' $  of  $ \, \uhg $  generated by all the  $ x'_i $'s,  where  $ \, I''_\hbar := \hbar\,{\uhg}'' + \textsl{Ker}\,\big(\,{\uhg}''\,\big) \, $.  In a nutshell,
  $ \; {\uhg}' = \Bbbk\big[\big[ {\{ x'_i \}}_{i \in I} \cup \{\hbar\} \big]\big] \; $.
%
\end{lemma}

\begin{proof}
 Let us consider an element  $ \, x' \in {\uhg}' \setminus \{0\} \, $.  Then there exist unique  $ \, n \in \N \, $  and  $ \, x \in \uhg \setminus h \, \uhg \, $  such that  $ \, x' = h^n x \, $;  \,applying  \cite[Lemma 3.3]{gavarini-02},  we find that for  $ \, \bar{x} := \big(\, x \!\mod \hbar\,\uhg \big) = U(\g) \, $  we have  $ \, \partial(\bar{x}) \leq n \, $  with respect to the canonical filtration of  $ U(\g) \, $,  \,that is  $ \bar{x} $  can be written as a linear combination of PBW monomials (w.r.t.\ any  $ \Bbbk $--basis  of  $ \g \, $)  whose degree are less or equal than  $ n \, $.  In particular, for the given  $ \Bbbk $--basis  $ {\{ \overline{x}_i \}}_{i \in I} $  of  $ \g $  we have  $ \, \bar{x} = P\big({\{\bar{x}_i\}}_{i \in I}\big) \, $  for some polynomial  $ P $  in the  $ \bar{x}_i $'s  with coefficients in  $ \Bbbk $  and degree  $ \, \partial(P) \leq n \, $.  Then
  $$  x_{\langle 0 \rangle}  \, := \,  P \big( {\{x_i\}}_{i \in I} \big)  \, \equiv \,  x \mod\, h \, \uhg  $$
 hence  $ \, x = x_{\langle 0 \rangle} + h \, x_{\langle 1 \rangle} \, $  for some  $ \, x_{\langle 1 \rangle} \in \uhg \, $,  \,and also  $ \, x' = \hbar^n x = \hbar^n x_{\langle 0 \rangle} + \hbar^{n+1} \, x_{\langle 1 \rangle} \, $.  Now we can write
  $$  \hbar^n x_0  \; = \;  \hbar^n P \big( {\{x_i\}}_{i \in I} \big)  \; = \;  P_0 \big( {\{\, \hbar\,x_i = : x'_i \,\}}_{i \in I} \big)  \; = \;  P_0 \big( {\{\, x'_i \,\}}_{i \in I} \big)  \; =: \;  x'_{(0)}  $$
 where  $ P_0 $  is again a polynomial (in the  $ x'_i $'s)  of degree bounded by  $ n \, $;  \,but then  $ \; x'_{(0)} \, := \, P_0 \big( {\{\, x'_i \,\}}_{i \in I} \big) \, \in \, {\uhg}' \; $  and the above yields
  $$  h^{n+1} x_{\langle 1 \rangle}  \, = \,  x' - x'_{(0)} \, \in \, {\uhg}'  $$
 Now, if  $ \, x'_{(1)} := h^{n+1} x_{\langle 1 \rangle} = 0 \, $  we are done; if not, we repeat the same argument with  $ x_{(1)} $  in the role of  $ \, x_{(0)} := x' \, $.  Iterating this procedure, we end up with a sequence  $ \, x'_{(k)} := P_k \big( {\{\, x'_i \,\}}_{i \in I} \big) \in {\uhg}' \, $,  $ \, k \in \N \, $,  \,where each  $ P_k $  is a polynomial in the  $ x'_i $'s  with coefficients in  $ \Bbbk \, $,  \,such that  $ \, x' =
\sum\limits_{k=0}^{+\infty} \hbar^k x'_{(k)} \, $  in  $ {\uhg}' \, $,  \,whence the claim.
\end{proof}

\vskip9pt

\begin{rem}  \label{rem: inf-dim case}
 A careful checking of all the analysis of Drinfeld's functors carried on in  \cite{gavarini-02}  show that they do make sense, and the results about them still hold true, also under the assumption that  $ \g $  be infinite dimensional, up to suitably (though slightly) adapting definitions to this more general setup.  Mainly, one has to properly choose the kind of topology, and relative completions, that one considers on the ``quantum groups''   --- as free  $ \kh $--modules  ---   under exam; accordingly, also the dual  $ \kh $--modules  (either  \textsl{full\/}  duals or  \textsl{topological\/}  ones) have to be chosen appropriately, and more choices are possible, indeed.  Note that these issues in fact do not depend on the ``quantum'' nature of the problem, but rather they are just plain translations of the same issues at the (semi)classical level, concerning the very definition and nature of  $ U(\g) $  and/or  $ F[[G]] $  when  $ \g $  and  $ G $  are infinite-dimensional.  Once we fix those issues, making appropriate choices, in the semiclassical setup, one has a canonical way to fix them, in a consistent way, in the quantum setup as well.
\end{rem}

\medskip

%
%
 \subsection{QDP for  $ \uhgx \, $:  the formal version}
 We deal now with a QUEA  $ \uhgx $  over a continuous Kac-Moody algebra  $ \g_X $  as in  \S \ref{subsec: quantiz_cont-KM-bialg's}.  We begin considering a new algebra whose definition is prompted by the description of  $ {\uhg}' $  in   \hbox{Proposition \ref{prop: exist-lifts-x_i}.}

\medskip

\begin{definition}  \label{def: Utildehgx}
 Given  $ \uhgx $  as above, we define  $ \utildehgx $  as follows.  Let  $ \dot{U}_\hbar(\g_X) $  be the unital  $ \kh $--subalgebra  of  $ \uhgx $  generated by all the elements
  $$  \barXi_\alpha := \hbar \, \varXi_\alpha  \;\; ,  \quad
   \barX{}_\alpha^{\,+} := \big(\, q - q^{-1} \big) \, X_\alpha^+  \;\; ,  \quad
  \barX{}_\alpha^{\,-} := \big(\, q - q^{-1} \big) \, X_\alpha^-  \quad   \eqno  \forall \;\; \alpha \in \intsf(X)  \quad  $$
 and let  $ I_\hbar $  is the two-sided ideal of  $ \dot{U}_\hbar(\g_X) $  generated by the  $ \barXi_\alpha $'s,  the  $ \barX{}_\alpha^{\,\pm} $'s  and  $ \, \hbar \cdot 1_{\uhgx} \, $:  \,then  $ \utildehgx $  by definition is the  $ I_\hbar $--adic  completion of  $ \dot{U}_\hbar(\g_X) \, $.
                                                                       \par
   Similarly, we define  $ \utildehbxp \, $,  resp.\  $ \utildehbxm \, $,  by the same procedure, but taking only the  $ \barXi_\ia $'s  and the  $ \barX{}_\alpha^{\,+} $'s,  resp.\ the  $ \barXi_\ia $'s  and the  $ \barX{}_\alpha^{\,-} $'s  ($ \, \ia \in \intsf(X) \, $),  as generators to deal with.  Clearly  $ \widetilde{U}_\hbar\big(\b_X^\pm\big) \, $,  with either sign, is also a subalgebra of  $ \utildehgx \, $.
\end{definition}

\vskip7pt

   The following is a direct consequence of  \S \ref{subsec: quantiz_cont-KM-bialg's},  Theorem \ref{thm: Hopf-str_Uhgx}  and  Definition \ref{def: Utildehgx}:

\vskip9pt

\begin{proposition}  \label{prop: pres-utildehgx}
 The algebra  $ \utildehgx $  is a Hopf  $ \kh $--subalgebra  of  $ \uhgx \, $.  Moreover, it admits the following presentation: it is the topological,  $ I_\hbar $--adically  complete Hopf  $ \kh $--algebra  with generators  $ \, \barXi_\ia \, $, $ \, \barX{}^{\,\pm}_\ib \, \big(\, \ia \, , \ib \in \intsf(X) \,\big) \, $  and relations
 (for all  $ \, \ia , \ib \in \intsf(X) \, $,  with in addition  $ \, (\ia,\ib) \in \serre{X} \, $  for the last relation)
%
  $$  \displaylines{
   \barXi_{\ia} \, \barXi_{\ib} - \barXi_{\ib} \, \barXi_{\ia} = \, 0  \; ,  \;\;\;
 \barXi_{\alpha \oplus \beta}  =  \delta_{\alpha \oplus \beta} \big(\, \barXi_\alpha + \barXi_\beta \big)  \; ,  \;\;\;
 \barXi_{\ia} \, \barX{}^{\,\pm}_\ib - \barX{}^{\,\pm}_\ib \, \barXi_{\ia}  = \,  \pm \hbar \rbf{\ia}{\ib} \barX{}^{\,\pm}_\ib  \cr
%
   \barX{}^{\,+}_\ia \, \barX{}^{\,-}_\ib - \barX{}^{\,-}_\ib \, \barX{}^{\,+}_\ia  \,\; = \;\,  \big(\, q - q^{-1} \big) \bigg( \drc{\ia \ib} \Big(\, \barK_\ia \! - \barK{}_\ia^{\,-1} \Big) \; +    \hfill  \cr
   + \; \ca{\ia}{\ib} \left(\, q^{\qcc{\ia}{\ib}{+}} \barX{}^{\,+}_{\!\ia\ominus\ib} \, \barK{}_\beta^{\,\ca{\ia}{\ib}} \! - q^{\qcc{\ia}{\ib}{-}} \barK{}_\alpha^{\,\ca{\ia}{\ib}} \barX{}^{\,-}_{\ib\ominus\ia} \right) \, + \; \qcb{\ib}{\ia}{} \, q^{\qcb{\ib}{\ia}{}} \, \barX{}^{\,+}_{(\iM{\ia}{\ib})\ominus\ib} \, \barK{}_{\im{\ia\, }{\, \ib}}^{\,\qcb{\ia}{\ib}{}} \, \barX{}^{\,-}_{(\iM{\ia}{\ib})\ominus{\ia}} \bigg)  \cr
%
%
   \barX{}^{\,\pm}_{\ia} \, \barX{}^{\,\pm}_{\ib} \, - \, q^{\qcr{\ia}{\ib}{}} \, \barX{}^{\,\pm}_{\ib} \, \barX{}^{\,\pm}_{\ia}  \; = \; \big(\, q - q^{-1} \big) \, \qcb{\ia}{\ib}{} \, \Big( \pm q^{\qcs{\ia}{\ib}{\pm}} \, \barX{}^{\,\pm}_{\ia\oplus\ib} \, + \, \barX{}^{\,\pm}_{\iM{\ia}{\ib}} \, \barX{}^{\,\pm}_{\im{\ia\, }{\,\ib}} \,\Big)  }  $$
 (notation of\/  \S \ref{subsec: quantiz_cont-KM-bialg's}),  where  $ \, \barK{}^{\,\pm 1}_\ia := \exp\big(\pm \barXi_\ia / 2 \big) = \exp\big(\pm \hbar \, \varXi / 2 \big) = K_\ia^{\pm 1} \, $  and  $ I_\hbar $  is the two-sided ideal generated by the  $ \barXi_\ia $'s  and  $ \barX{}^{\,\pm}_\ib $'s,  \,with Hopf structure given by
  $$  \displaylines{
   \epsilon\big(\,\barXi_\ia\big) \, = \, 0 \, = \, \epsilon\big(\,\barX{}^{\,\pm}_\ia\big)  \quad ,   \qquad \qquad  \Delta\big(\,\barXi_\ia\big)\ := \barXi_\ia \otimes 1 + 1 \otimes \barXi_\ia  \cr
   \Delta\big(\,\barX{}^{\,+}_\ia\big)  \; = \;  \barX{}^{\,+}_\ia \otimes 1 + \barK{}_\ia^{\,+1} \otimes \barX{}^{\,+}_\ia \, + \, {\textstyle \sum_{\ia = \ib \oplus \ic}} \, \ca{\ib}{, \,\ia} \, \qcs{\ic}{\ib}{+} \cdot \, \barK{}_\ic^{\,+1} \, \barX{}^{\,+}_\ib \otimes \barX{}^{\,+}_\ic  \cr
   \Delta\big(\,\barX{}^{\,-}_\alpha\big)  \, = \,  1 \otimes \barX{}^{\,-}_\alpha + \barX{}^{\,-}_\alpha \otimes \barK{}_\alpha^{\,-1} \, - \, {\textstyle \sum_{\alpha = \beta \oplus \gamma}} \, \ca{\ib}{, \,\ia} \, \qcs{\ic}{\ib}{-} \cdot \barX{}^{\,-}_\beta \otimes \barX{}^{\,-}_\ic \, \barK{}_\ib^{\,-1}  }  $$
 --- where  $ \Delta $  takes values in the completion of the algebraic tensor product with respect to the ``weak \&  $ \hbar $--adic  topology'' ---   with antipode
 $ \; S := \sum_n m^{(n)} \circ {(\id - \iota \circ \epsilon)}^{\otimes n} \circ \Delta^{(n)} \, $.
 \vskip3pt
   Similarly,  $ \utildehbxp \, $,  resp.\  $ \utildehbxm \, $,  is a Hopf  $ \, \kh $--subalgebra  of  $ \, \uhbxp \, $,  resp.\  of  $ \, \uhbxm \, $,  and it admits an analogous presentation by generators and relations.
\end{proposition}

\vskip7pt

   As a direct consequence of the previous result, we get the following, which is nothing but an appropriate version of the QDP for the QUEA  $ \uhgx \, $:

\vskip9pt

\begin{theorem}  \label{thm: QDP x Uhgx}
 The topological Hopf algebra  $ \utildehgx $  is a QFSHA, which is a quantization of the formal Poisson group  $ G^*_{\!X} $  dual to the Lie bialgebra  $ \g_X \, $.
                                         \par
   More in detail, the following holds.  For  $ \; \utildezgx := \utildehgx \Big/ \hbar\,\utildehgx \; $, \,we have:
 \vskip-1pt
   (a)\;  $ \; \utildezgx \, $  is a  \textsl{commutative}  (topological) Hopf algebra;
 \vskip3pt
   (\,b)\;  $ \; \utildezgx \, $  is  $ \, I $--adically  complete, with  $ \; I := \textsl{Ker}\,\big(\, \epsilon : \utildezgx \relbar\joinrel\relbar\joinrel\relbar\joinrel\longrightarrow \Bbbk \,\big) \, $  where  $ \epsilon $  is the augmentation map of the Hopf algebra  $ \utildezgx \, $;
 \vskip3pt
   (\,c)\;  the Lie bialgebra structure on  $ \, I \big/ I^2 \, $  from  $ \utildehgx $   --- following  \S \ref{subsec: quant_Lie-bialg_&_P-groups}  ---   makes  $ \, I \big/ I^2 \, $  into a Lie  $ \, \Bbbk $--bialgebra  isomorphic to  $ \, \g_X \, $,  given by (for all  $ \, \ia \in \intsf(X) \, $)
  $$  \Big(\; \barXi_\ia \; \big(\, \text{\rm mod\ } \hbar\,\utildehgx \big) \;\; \text{\rm mod\ } I^2 \,\Big) \mapsto \, \xi_\ia  \; ,   \quad
   \Big(\; \barX_\ia^{\,\pm} \; \big(\, \text{\rm mod\ } \hbar\,\utildehgx \big) \;\; \text{\rm mod\ } I^2 \,\Big) \mapsto \, x_\ia^\pm  $$
\end{theorem}

\begin{proof}
 Claim  \textit{(a)\/}  follows at once from the commutation relations in  Proposition \ref{prop: pres-utildehgx},  since  $ \, q \equiv 1  \mod \hbar\,\kh \, $  by construction.
 \vskip3pt
   Similarly, claim  \textit{(b)\/}  follows from the fact that  $ \utildehgx $  is  $ I_\hbar $--adically  complete.
 \vskip3pt
   Finally, claim  \textit{(c)\/}  is a matter of sheer bookkeeping.  Indeed, it is clear by construction that the assignment in claim  \textit{(c)\/}  yields a  $ \Bbbk $--linear  isomorphism from  $ \, I \big/ I^2 \, $  to  $ \g_X \, $;  \,besides, one has just to check, tracking all definitions, that the recipe for the Lie bialgebra structure detailed in  \S \ref{subsec: quant_Lie-bialg_&_P-groups}  actually does provide for the generators of  $ \, I \big/ I^2 \, $  the very formulas that describe the Lie bracket and the Lie cobracket in  $ \g_X $  for the corresponding (following the assignment in claim  \textit{(c)}  ---   generators of  $ \g_X \, $.
                                                                      \par
   To give an insight, we show how to prove a couple of instances, all other cases being similar.  Using notation  $ \, \widetilde{U}_\hbar := \utildehgx \, $  and  $ \; {\check{X}}_\ic^\pm := \Big(\, \barX{}_\ic^{\,\pm} \, \big(\, \text{\rm mod\ } \hbar\,\widetilde{U}_\hbar \big) \,\; \text{\rm mod\ } I^2 \,\Big) \; $,  \;we show that
\begin{equation}  \label{eq: map-Lie-bracket}
   \big[\, {\check{X}}_\ia^+ \, , {\check{X}}_\ib^- \,\big]  \,\; \mapsto \;\,  \big[\, x_\ia^+ , x_\ib^- \,\big]
\end{equation}
 Indeed, by construction   --- setting  $ \; \check{\varXi}_\ia := \Big(\; \barXi_\ia \, \big(\, \text{\rm mod\ } \hbar\,\widetilde{U}_\hbar \big) \,\; \text{\rm mod\ } I^2 \,\Big) \, $  ---   we have
  $$  \displaylines{
   \big[\, {\check{X}}_\ia^+ \, , {\check{X}}_\ib^- \,\big]  \; = \;  \Big\{ \Big(\, \barX{}_\ia^{\,+} \; \text{\rm mod\ } \hbar\,\widetilde{U}_\hbar \Big) \, , \Big(\, \barX{}_\ib^{\,-} \; \text{\rm mod\ } \hbar\,\widetilde{U}_\hbar \Big) \Big\} \; \text{\rm mod\ } I^2  \; =   \hfill  \cr
   = \;  \Bigg(\, \frac{\,\big[\, \barX{}_\ia^{\,+} , \barX{}_\ib^{\,-} \,\big]\,}{\hbar} \; \Big(\, \text{\rm mod\ } \hbar\,\widetilde{U}_\hbar \Big) \Bigg) \; \text{\rm mod\ } I^2  \; =   \hfill  \cr
   \hfill   = \;  \Bigg(\, \frac{\, \barX{}_\ia^{\,+} \, \barX{}_\ib^{\,-} - \barX{}^{\,-}_\ib \, \barX{}^{\,+}_\ia \,}{\hbar} \; \Big(\, \text{\rm mod\ } \hbar\,\widetilde{U}_\hbar \Big) \Bigg) \; \text{\rm mod\ } I^2  \; =  \cr
   = \;  \Bigg( \Bigg(\, \frac{\; q - q^{-1} \,}{\hbar} \, \bigg(\, \drc{\ia \ib} \Big(\, \barK{}_\ia^{\,+1} \! - \barK{}_\ia^{\,-1} \Big) + \ca{\ia}{\ib} \left(\, q^{\qcc{\ia}{\ib}{+}} \barX{}^{\,+}_{\!\ia\ominus\ib} \, \barK{}_{\ib}^{\,\ca{\ia}{\ib}} \! - q^{\qcc{\ia}{\ib}{-}} \barK_{\ia}^{\,\ca{\ia}{\ib}} \barX{}^{\,-}_{\ib\ominus\ia} \,\right)  \, +   \hfill  \cr
   \hfill   + \;\;  \qcb{\ib}{\ia}{} \, q^{\qcb{\ib}{\ia}{}} \, \barX{}^{\,+}_{(\iM{\ia}{\ib})\ominus\ib} \, \barK{}_{\im{\ia\,}{\,\ib}}^{\,\qcb{\ia}{\ib}{}} \, \barX{}^{\,-}_{(\iM{\ia}{\ib})\ominus{\ia}} \,\bigg) \Bigg) \;\; \text{\rm mod\ } \hbar\,\widetilde{U}_\hbar \,\Bigg) \;\; \text{\rm mod\ } I^2  \; =  \cr
   = \;  \Bigg( \bigg(\, \drc{\ia \ib} \Big(\, \barK{}_\ia^{\,+1} \! - \barK{}_\ia^{\,-1} \Big) + \ca{\ia}{\ib} \left(\, \barX{}^{\,+}_{\!\ia\ominus\ib} \, \barK{}_{\ib}^{\,\ca{\ia}{\ib}} \! - \barK{}_{\ia}^{\,\ca{\ia}{\ib}} \barX{}^{\,-}_{\ib\ominus\ia} \,\right)  \, +   \hfill  \cr
   \hfill   + \;\;  \qcb{\ib}{\ia}{} \, \barX{}^{\,+}_{(\iM{\ia}{\ib})\ominus\ib} \, \barK{}_{\im{\ia\, }{\, \ib}}^{\,\qcb{\ia}{\ib}{}} \, \barX{}^{\,-}_{(\iM{\ia}{\ib})\ominus{\ia}} \,\bigg) \;\; \text{\rm mod\ } \hbar\,\widetilde{U}_\hbar \,\Bigg) \;\; \text{\rm mod\ } I^2  \; =  \cr
   = \;  \Bigg( \bigg(\, \drc{\ia \ib} \Big(\, \barK{}_\ia^{\,+1} \! - \barK{}_\ia^{\,-1} \Big) + \ca{\ia}{\ib} \left(\, \barX{}^{\,+}_{\!\ia\ominus\ib} \, \barK{}_{\ib}^{\,\ca{\ia}{\ib}} \! - \barK{}_{\ia}^{\,\ca{\ia}{\ib}} \barX{}^{\,-}_{\ib\ominus\ia} \,\right) \!\bigg) \;\; \text{\rm mod\ } \hbar\,\widetilde{U}_\hbar \,\Bigg) \;\; \text{\rm mod\ } I^2  \; =   \hfill  \cr
   \qquad   = \;  \Bigg( \bigg(\, \drc{\ia \ib} \, \barXi_\ia + \ca{\ia}{\ib} \left(\, \barX{}^{\,+}_{\!\ia\ominus\ib} - \barX{}^{\,-}_{\ib\ominus\ia} \,\right) \!\bigg) \;\; \text{\rm mod\ } \hbar\,\widetilde{U}_\hbar \,\Bigg) \;\; \text{\rm mod\ } I^2  \; =   \hfill  \cr
   \hfill   = \;  \bigg(\, \drc{\ia \ib} \, \check{\varXi}_\ia + \, \ca{\ia}{\ib} \left(\, \check{X}{}^{\,+}_{\!\ia\ominus\ib} - \check{X}{}^{\,-}_{\!\ib\ominus\ia} \,\right) \!\!\bigg)  \;\; \mapsto \;\;  \drc{\ia \ib} \, \xi_\ia \, + \, \ca{\ia}{\ib} \left(\, x^{\,+}_{\!\ia\ominus\ib} - x^{\,-}_{\ib\ominus\ia} \,\right)  \,\; = \;\,  \big[\, x_\ia^+ , x_\ib^- \,\big]  }  $$
 so that  \eqref{eq: map-Lie-bracket}  is proved.  Similarly, for the Lie cobracket we go and prove that
\begin{equation}  \label{eq: map-Lie-cobracket}
   \delta\big(\check{X}{}^{\,+}_\ia\big)  \; \mapsto \;  \delta\big(x^+_\ia\big)
\end{equation}
 First of all, definitions give
  $$  \delta\big(\check{X}{}^{\,+}_\ia\big)  \;\; = \;\;  \Big( \big(\Delta - \Delta^{\text{op}} \big)\big(\, \check{X}{}^{\,+}_\ia \,\big)  $$
 and, setting  $ \; I_2 \, := \, I^2 \otimes I + I \otimes I^2 \, $,  \,direct computations yield
 $$  \displaylines{
  \Delta\big(\check{X}{}^{\,+}_\ia\big)  \,\; = \;\,  \Big(\, \Delta\big(\,\barX{}^{\,+}_\ia\big) \;\; \text{\rm mod\ } \hbar\,\widetilde{U}_\hbar^{\,\otimes 2} \,\Big) \;\; \text{\rm mod\ } I_2  \,\; =   \hfill  \cr
  = \,  \bigg(\! \Big(\, \barX{}^{\,+}_\ia \otimes 1 + \barK_\ia \otimes \barX{}^{\,+}_\ia \, + \hskip-6pt {\textstyle \sum\limits_{\;\ia = \ib \oplus \ic}} \hskip-4pt \ca{\ib}{, \,\ia} \, \qcs{\ic}{\ib}{+} \, \barK_\ic \, \barX{}^{\,+}_\ib \otimes \barX{}^{\,+}_\ic \Big) \; \text{\rm mod\ } \hbar\,\widetilde{U}_\hbar^{\,\otimes 2} \bigg) \; \text{\rm mod\ } I_2  \, =  \cr
  \hfill   = \;\,  \check{X}{}^{\,+}_\ia \otimes 1 \, + \, \big( 1 + \check{\varXi}_\ia / 2 \big) \otimes \check{X}{}^{\,+}_\ia \, + \, {\textstyle \sum_{\ia = \ib \oplus \ic}} \, \ca{\ib}{,\,\ia} \, \qcs{\ic}{\ib}{+} \, \check{X}{}^{\,+}_\ib \otimes \check{X}{}^{\,+}_\ic  }  $$
 whence one gets (exploiting a careful analysis   --- that is also necessary to prove the parallel statement when proving part  \textit{(5)\/}  of  Theorem \ref{thm: Hopf-str_Uhgx}  ---   of the values of the coefficients  $ \ca{\ib}{,\,\ia} $'s  and  $ \qcs{\ic}{\ib}{+} $'s,
   as in  \cite{appel-sala-20})  that\footnote{\,Hereafter we write  $ \; a \wedge b \, := \, (ab-ba)\big/2 \; $  for any elements  $ \, a, b \, $  in any algebra  $ A \, $.}
  $$  \displaylines{
   \quad   \delta\big(\check{X}{}^{\,+}_\ia\big)  \; = \;  \big(\Delta - \Delta^{\text{op}} \big)\big(\check{X}{}^{\,+}_\ia\big)  \; = \;  \check{\varXi}_\ia \wedge \check{X}{}^{\,+}_\ia \, + \, {\textstyle \sum_{\ia = \ib \oplus \ic}} \, \ca{\ib}{, \,\ia} \, \qcs{\ic}{\ib}{+} \, 2 \, \check{X}{}^{\,+}_\ib \!\wedge \check{X}{}^{\,+}_\ic  \;\; \mapsto   \hfill  \cr
   \hfill   \mapsto \;\;  \xi_\ia \wedge x^{\,+}_\ia \, + \, {\textstyle \sum_{\ia = \ib \oplus \ic}} \, \ca{\ib}{, \,\ia} \, x^{\,+}_\ib \!\wedge x^{\,+}_\ic  \; = \;  \delta\big(x^+_\ia\big)   \quad  }  $$
 so that  \eqref{eq: map-Lie-cobracket}  is proved.
\end{proof}

\medskip

\subsection{Formal QDP for  $ \uhgx \, $:  the intrinsic recipe}
 In the previous subsection, we introduced a subalgebra  $ \utildehgx $  of  $ \uhgx \, $,  then we proved that it is actually a Hopf algebra, and even a QFSHA that is a quantization of  $ G_X^* \, $,  the formal Poisson group dual to  $ \g_X \, $.  In this sense, we have realized the QDP for the QUEA  $ \uhgx \, $.
                                                            \par
   The explicit construction of  $ \uhgx $  is inspired by similar constructions for the standard, well-known \textsl{polynomial\/}  QUEA  $ U_q(\g) $  associated with a Kac-Moody algebra  $ \g $  of finite or affine type following Drinfeld and Jimbo  (cf.\ \cite{gavarini_PJM-98,gavarini-00},  and references therein): in these cases, the explicit construction amounts to taking, within the given QUEA, the subalgebra generated by renormalized lifts of vectors in  $ \g \; \big( \subseteq U(\g) \big) \, $.  In the present framework one also has the additional advantage that  $ \uhgx $  is already endowed, by definition, with suitable built-in ``quantum root vectors''  $ X_\ia^\pm $  that are lifts of the ``root vectors''  $ x_\ia^\pm $  in  $ \g_X \, $:  this makes things easier than in the case of finite or affine Kac-Moody  $ \g \, $.
                                                            \par
   It is explained in  \cite{gavarini-02}  that even the general construction provided by the QDP in its full extent can still be realized in this way: however, there exists no way whatsoever for making this rough idea into a precise recipe.  On the other hand, the core formulation of the QDP for QUEA does provide an explicit recipe for the like of  $ \utildehgx \, $,  given in intrinsic terms.  We shall now show that the same definition makes sense for the QUEA  $ \uhgx \, $,  and in fact provides a different realization of the same Hopf subalgebra  $ \utildehgx $  that we considered above.

\vskip11pt

\begin{definition}  \label{def: U'hgx - BIS}
 \textsl{(cf.\ Definition \ref{def: Drinfeld's functors}\textit{(a)})}
 Let  $ \uhgx $  be as in  \S \ref{subsec: quantiz_cont-KM-bialg's},  and let also  $ \, \iota : \kh \relbar\joinrel\longrightarrow \uhgx \, $  and  $ \, \epsilon : \uhgx \relbar\joinrel\longrightarrow \kh \, $  be its unit and counit maps; moreover, for every  $ \, n \in \N \, $  set  $ \, \delta_n := (\id - \iota \circ \epsilon) \circ \Delta^{(n)} \, $.  Then we define
  $$  {\uhgx}'  \; := \;  \Big\{\, \eta \in \uhgx \,\Big|\; \delta_n(\eta) \in \hbar^n\,{\uhgx}^{\otimes n} \;\; \forall \; n \in \N \,\Big\}  $$
\end{definition}

\vskip1pt

   The key point now is our next result:

\vskip11pt

\begin{theorem}  \label{thm: uhgx'=utildehgx}
 With notation as before, we have  $ \,\; \utildehgx \, = \, {\uhgx}' \;\, $.
\end{theorem}

\begin{proof}
 By direct check, one finds at once that  $ \; \varXi_\ia \, , X_\ia^\pm \not\in {\uhgx}' \; $  but  $ \; \hbar\,\varXi_\ia \, , \hbar\,X_\ia^\pm \in {\uhgx}' \; $;  \,hence, once we note that  $ \, \big(\, q - q^{-1} \big) = \hbar \, \kappa \; $  with  $ \kappa $  an invertible element in  $ \kh \, $,  \,we get
 $ \; \barXi_\ia \, , \barX{}_\ia^\pm \in {\uhgx}' \; $  for all  $ \, \ia \in \intsf(X) \, $.
 Moreover, by the properties of the maps  $ \delta_n \, (\, n \in \N \,) \, $  explained in  \cite{gavarini-02},  one sees easily that  $ \, {\uhgx}' \, $  is a unital  $ \kh $--subalgebra  of  $ \uhgx \, $.  Therefore, we conclude that  $ \; \utildehgx \subseteq {\uhgx}' \; $.
                                                                \par
   As to the converse inclusion, first note that  $ \, \mathbb{B} := {\big\{ x_\ia^+ \, , \xi_\ia \, , x_\ia^- \big\}}_{\ia \in \intsf(X)} \, $  is a  $ \Bbbk $--spanning  set for  $ \g_X \, $:  in fact, it falls short from being a  $ \Bbbk $--basis  only because of the relations
%
%
 $ \; \xi_{\alpha \oplus \beta} \, = \, \delta_{\alpha \oplus \beta} (\xi_\alpha + \xi_\beta) \; $   --- for  $ \, \alpha , \beta \in \intsf(X) \, $  ---   but this will not affect our argument hereafter.  Now, the elements  $ \varXi_\ia $'s  and  $ X_\ia^\pm $'s  (for all  $ \, \alpha \in \intsf(X) \, $)  are lifts of the  $ \xi_\ia $'s  and  $ x_\ia^- $'s  (respectively) in  $ \, \textsl{Ker\/}\big(\uhgx\big) \cap\, {\uhgx}' \; $;  \,then  Lemma \ref{lemma: generation-ughx'-by-x'_i}  applies   --- even in the present, slightly modifed and infinite-dimensional setting ---   so that  $ {\uhgx}' $  is indeed the algebra of ``formal series''
  $$  {\uhgx}'  \, = \;  \Bbbk\big[\big[ {\big\{ X_\ia^+ \, , \varXi_\ia \, , X_\ia^- \big\}}_{\ia \in \intsf(X)} \cup \{\hbar\} \big]\big]  \; \subseteq \;  \utildehgx  $$
 whence (by the first part of the proof) we get  $ \; \utildehgx = {\uhgx}' \; $,  \;q.e.d.
\end{proof}

\bigskip

\section{QDP for polynomial quantum continuous KM-algebras}
 In this section we provide a ``polynomial'' version of the QDP, i.e.\ a version that applies to a ``polynomial'' QUEA, much in the spirit of  \cite{gavarini-07}  instead of  \cite{gavarini-02}.  Like in  \S \ref{sec: QDP-formal},  we cannot directly apply the results in  \cite{gavarini-07},  but we rather have to cook up a suitable recipe that might fit the present case; this will be done by following closely the pattern offered by  \cite{beck-kac}   --- see also  \cite[Theorem 6.3]{gavarini-00}  for an alternative approach.

\medskip

\subsection{The polynomial QUEA  $ \uqgx \, $}  \label{subsec: polyn-QUEA}
 Besides the original formulation by Drinfeld (cf.\ \cite{drinfeld-quantum-groups-87}  and  \cite{gavarini-02})  in the framework of ``formal'' quantization, another version of the QDP was developed in  \cite{gavarini-07}  for the setup of ``polynomial'' quantization, that is for genuine Hopf algebras over some ground ring  $ R $  that for at some special quotient of  $ R $  itself specialize (in a suitable sense) to a the universal enveloping algebra of some Lie (bi)algebra.
                                           \par
   In this spirit, our first step is to ``extract'' from ``formal'' QUEA  $ \uhgx $  a ``polynomial'' QUEA in the previous sense, as follows:

\medskip

\begin{definition}  \label{def: polyn_Uqgx}
 Let  $ \, q^{\pm 1} := \exp\big(\!\pm\hbar/2\big) \in \kh \, $,  \,and let  $ \kqqm $  be the corresponding  $ \Bbbk $--algebra  of Laurent polynomials in  $ q $  embedded into  $ \kh \, $.  Given a (formal) QUEA  $ \uhgx $  as in  \S \ref{subsec: quantiz_cont-KM-bialg's},  we define the  \textsl{polynomial QUEA}  $ \uqgx $
   as being\,\footnote{Indeed, a more ``refined'' definition is possible; the present one is the simplest possible (carefully choosing normalizations), which is enough for our present scopes.}
 the unital  $ \kqqm $--subalgebra  of  $ \uhgx $  generated by the elements
  $$  \dot{K}_\alpha^{\pm 1} := K_\alpha^{\pm 1} \;\; ,  \quad
     \dot{H}_\alpha := {{\, K_\alpha - 1 \,} \over {\, q - 1 \,}}  \;\; ,  \quad
    \dot{X}_\alpha^\pm := \big( 1 + q^{-1} \big) X_\alpha^\pm   \eqno \forall \;\; \ia \in \intsf(X)  \quad  $$
   \indent   Similarly, we also consider the unital  $ \kqqm $--subalgebras  $ \uqbxpm $  in  $ \uhgx $   --- hence in $ \uqgx $  ---   generated by the  $ \qxz{\ia}{} $'s  and the  $ \qxpm{\ia} $'s  (for  $ \, \ia \in \intsf(X) \, $).
\end{definition}

\medskip

   The following result yields a description of  $ \uqgx \, $:

\medskip

\begin{proposition}  \label{prop: Hopf-str_Uqgx}
 Let  $ \calQ_X $  be a continuous quiver and  $ \, \uqgx $  the corresponding  $ \kqqm $--algebra  given in Definition \ref{def: polyn_Uqgx}  above.  Then:
 \vskip5pt
   (1)\,  $ \uqgx $  has the following presentation: it is the unital, associative  $ \kqqm $--algebra  with generators  $ \dot{K}_\alpha^{\pm 1} $,  $ \dot{H}_\alpha $  and  $ \dot{X}_\alpha^\pm $  (for all  $ \, \ia \in \intsf(X) \, $)  and relations
 (for all  $ \, \ia , \ib \in \intsf(X) \, $,  with the additional constraint  $ \, (\ia,\ib) \in \serre{X} \, $  for the last relation)
%
  $$  \displaylines{
   \dot{K}_\alpha^{\pm 1} \, \dot{K}_\beta^{\pm 1} = \, \dot{K}_\beta^{\pm 1} \, \dot{K}_\alpha^{\pm 1}  \;\; ,
\quad   \dot{K}_\alpha^{\pm 1} \, \dot{K}_\beta^{\mp 1} = \, \dot{K}_\beta^{\mp 1} \, \dot{K}_\alpha^{\pm 1}  \;\; ,  \quad
  \dot{K}_\alpha^{\pm 1} \, \dot{K}_\alpha^{\mp 1} = \, 1  \cr
   \dot{K}_{\alpha \oplus \beta}^{\pm 1}  = \,  \delta_{\alpha \oplus \beta} \, \dot{K}_\alpha^{\pm 1} \dot{K}_\beta^{\pm 1}  \quad ,  \quad \qquad
  \dot{K}_\alpha  \, = \,  1 \, + \, (\,q - 1) \, \dot{H}_\alpha  \cr
   \dot{H}_{\alpha \oplus \beta}  = \,  \delta_{\alpha \oplus \beta} \big( \dot{H}_\alpha \dot{K}_\beta + \dot{H}_\beta \big)  \quad ,  \quad \qquad  \dot{H}_\alpha \, \dot{H}_\beta \, = \, \dot{H}_\beta \, \dot{H}_\alpha  \cr
%
   \dot{K}_\alpha^{+1} \, \dot{X}_\beta^\pm \, \dot{K}_\alpha^{-1} \, = \; q^{\pm(\alpha|\,\beta)} \dot{X}_\beta^\pm  \quad ,   \qquad
  \dot{H}_\alpha \, \dot{X}_\beta^\pm \, - \, q^{\pm(\alpha|\,\beta)} \dot{X}_\beta^\pm \, \dot{H}_\alpha  \; = \; {\big(\!\pm(\alpha|\,\beta)\big)}_{\!q} \, \dot{X}_\beta^\pm
 }  $$
  $$  \displaylines{
   \dot{X}^+_{\ia} \, \dot{X}^-_{\ib} - \dot{X}^-_{\ib} \, \dot{X}^+_{\ia}  \,\; = \;\,  \delta_{\alpha{}\beta} \, \big(\, 1 + q^{-1} \big) \big(\, 1 + \dot{K}_\alpha^{-1} \big) \, \dot{H}_\alpha  \; +
\; \ca{\ia}{\ib} \, \big(\, 1 + q^{-1} \big) \, \times   \hfill  \cr
   \hfill   \times \, \left( q^{\qcc{\ia}{\ib}{+}} \dot{X}^+_{\ia\ominus\ib} \dot{K}_{\ib}^{\,\ca{\ia}{\ib}} \! - q^{\qcc{\ia}{\ib}{-}} \dot{K}_{\ia}^{\,\ca{\ia}{\ib}} \dot{X}^-_{\ib\ominus\ia} \right)  +  \qcb{\ib}{\ia}{} q^{\qcb{\ib}{\ia}{}} \big( q - q^{-1} \big) \dot{X}^+_{(\iM{\ia}{\ib})\ominus\ib} \, \dot{K}_{\im{\ia\, }{\, \ib}}^{\,\qcb{\ia}{\ib}{}} \dot{X}^-_{(\iM{\ia}{\ib})\ominus{\ia}}  \cr
   \dot{X}^\pm_{\ia} \, \dot{X}^\pm_{\ib} \, - \,  q^{\qcr{\ia}{\ib}{}} \dot{X}^\pm_{\ib} \, \dot{X}^\pm_{\ia}  \; = \;  \pm \, \qcb{\ia}{\ib}{} \, q^{\qcs{\ia}{\ib}{\pm}} \, \big(\, 1 + q^{-1} \big) \, \dot{X}^\pm_{\ia\oplus\ib} \, + \, \qcb{\ia}{\ib}{} \big( q - q^{-1} \big) \, \dot{X}^\pm_{\iM{\ia}{\ib}} \, \dot{X}^\pm_{\im{\ia\, }{\,\ib}}  }  $$
 where
 $ \; \displaystyle{ {\big(\!\pm(\alpha|\,\beta)\big)}_{\!q} := {{\; q^{\pm(\alpha|\,\beta)} - 1 \;} \over {\,q - 1\,}} } \in \kqqm \, $
 following standard  $ q $--number  notation and for the rest we use notation as in  \S \ref{subsec: quantiz_cont-KM-bialg's}.
 \vskip5pt
   (2)\,  $ \uqgx $  is Hopf subalgebra (over  $ \kqqm $)  of  $ \, \uhgx \, $,  \,with
  $$  \displaylines{
 \varepsilon\big(\dot{K}_\alpha^{\pm 1}\big) = 1  \; ,   \;\;  \varepsilon\big(\dot{H}_\alpha\big) = 0  \; ,   \;\;   \Delta\big(\dot{K}_\alpha^{\pm 1}\big) \, = \, \dot{K}_\alpha^{\pm 1} \otimes \dot{K}_\alpha^{\pm 1}  \; ,   \;\;   \Delta\big(\dot{H}_\alpha\big) \, = \, \dot{H}_\alpha \otimes \dot{K}_\alpha^{+1} + 1 \otimes \dot{H}_\alpha  \cr
   \epsilon\big(\dot{X}_\alpha^+\big) = \, 0  \; ,   \;\;  \Delta\big(\dot{X}^+_{\ia}\big)  = \,  \dot{X}^+_{\ia} \otimes 1 + \dot{K}_\ia \otimes \dot{X}^+_{\ia} + \hskip-5pt {\textstyle \sum\limits_{\ia = \ib \oplus \ic}} \hskip-3pt \ca{\ib}{, \,\ia} \, \qcs{\ic}{\ib}{+} \, (\,q-1) \, \dot{K}_\ic^{+1} \dot{X}^+_{\ib} \otimes \dot{X}^+_{\ic}  \cr
   \epsilon\big(\dot{X}_\alpha^-\big) = \, 0  \; ,   \;\;  \Delta\big(\dot{X}^-_{\ia}\big)  = \,  1 \otimes \dot{X}^-_{\ia} + \dot{X}^-_{\ia} \!\otimes \dot{K}_\ia^{-1} - \hskip-7pt {\textstyle \sum\limits_{\ia = \ib \oplus \ic}} \hskip-3pt \ca{\ib}{, \,\ia} \, \qcs{\ic}{\ib}{-} \, (\,q-1) \, \dot{X}^-_{\ib} \otimes \dot{X}^-_\ic \dot{K}_\ib^{-1}  }  $$
 while the antipode is given by the formula
 $ \; S := \sum_n m^{(n)} \circ {(\id - \iota \circ \epsilon)}^{\otimes n} \circ \Delta^{(n)} \; $.
 \vskip5pt
   (3)\,  $ \uqbxp \, $,  resp.\  $ \uqbxm \, $,  \,is a Hopf  $ \, \kqqm $--subalgebra  of  $ \, \uqgx $   --- hence of  $ \, \uhgx $  as well.
\end{proposition}

\vskip11pt

   The proof of the previous result is a sheer exercise of re-writing, just looking at definitions together with  Theorem \ref{thm: Hopf-str_Uhgx}.  Moreover, claims  \textit{(3)},  \textit{(4)\/}  and  \textit{(5)\/}  of  Theorem \ref{thm: Hopf-str_Uhgx}  induce similar results for  $ \uqbxpm $  as well, but we do not need them.
                                                           \par
   Next result instead explains why we call  $ \uqgx $  a QUEA again, though polynomial; the proof, again, is straightforward, but first we have to set the framework.

\vskip13pt

\begin{definition}  \label{def: polyn_QUEA}
 We define  \emph{(polynomial) quantized universal enveloping algebra (=QUEA) over\/  $ \kqqm $}  any Hopf\/  $ \kqqm $--algebra  $ U_q $  (in the standard sense) such that the quotient Hopf algebra over  $ \; \kqqm \big/ (\,q-1)\,\kqqm \, \cong \, \Bbbk \; $  given by  $ \,\; U_1 \, := \, U_q \big/ (\,q-1)\,U_q \;\, $  is connected and cocommutative   --- or, equivalently,  $ \, U_1 \, $  is isomorphic to an enveloping algebra  $ U(\g) $	for some Lie algebra  $ \g \, $.  For any such QUEA, the formula
 $ \; \displaystyle{ \delta(x) \, := \, \frac{\; \Delta\big(x'\big) - \Delta^{21}\big(x'\big) \;}{q-1} \; \mod (\,q-1)\,U_q^{\,\otimes 2} } \; $
 --- where  $ \, x' \in U_q \, $  is any lift of  $ \, x \in \g \, $  ---   yields a co-Poisson structure on  $ \, U_1 = U(\g) \, $,  \,hence a Lie bialgebra structure on  $ \g \, $.
                                                                      \par
   In this case, we say that  $ U_q $  is a  \emph{(polynomial) quantization\/}  of the co-Poisson Hopf algebra  $ U(\g) \, $,  or (with a slight abuse of language) of the Lie bialgebra  $ \g \, $.
\end{definition}

\vskip7pt

   The previous definition now gives sense to the following result:

\vskip11pt

\begin{theorem}  \label{thm: specializ-Uqgx}
 The Hopf algebra  $ \uqgx $  is a polynomial QUEA, and more precisely it is a polynomial quantization of the Lie bialgebra  $ \g_X \, $,  \,as there is an isomor\-phism  $ \,\; \uqgx \Big/ (\,q-1)\,\uqgx \,\cong\, U(\g_X) \;\, $  of
co-Poisson Hopf algebras given by
  $$  \dot{K}_\alpha^{\pm 1} \mapsto\, 1 \;\; ,  \quad  \dot{H}_\alpha \,\mapsto\, \xi_\alpha \;\; ,  \quad  \dot{X}_\alpha^\pm \,\mapsto\, 2\,x_\alpha^\pm   \qquad \qquad  \forall \;\; \alpha \in \intsf(X)  $$
   \indent   In a similar sense,  $ \uqbxpm $  is a quantization of the Lie bialgebra  $ \b_X^\pm \, $.
\end{theorem}

\begin{proof}
 This is a straightforward consequence of the very presentation of  $ \uqgx $  given in  Proposition \ref{prop: Hopf-str_Uqgx}  above.
\end{proof}

\vskip13pt

\begin{Observation}  \label{obs: uqgx over Zqqm}
 It is worth remaking that in the formulas occurring in  Proposition \ref{prop: Hopf-str_Uqgx}  \textsl{all coefficients that show up actually belong to the (sub)ring\/}  $ \Z\big[\,q,q^{-1}\big] \, $.  This implies that the Hopf algebra  $ \uqgx $  can be defined  \textsl{over\/}  $ \Z\big[\,q,q^{-1}\big] \, $,  \,hence one can go and consider  $ R $--integral  forms of it over any ring  $ R \, $,  \,then possibly look for specializations at roots of unity, etc.\   --- much like in the study of quantum groups by Lusztig and others, including (in an infinite dimensional setting, such as is that of  $ \g_X $)  the study in \cite{gavarini-00}  for the case of quantum  \textsl{dual\/}  affine Kac--Moody algebras.
\end{Observation}

\medskip

%
%
 \subsection{QDP for  $ \uqgx \, $:  the polynomial version}  \label{subsec: QDP-polyn}
 We are now ready to present our realization of the QDP for the polynomial QUE  $ \, \uqgx $  introduced above.

\medskip

\begin{definition}  \label{def: Utildeqgx}
 Given  $ \uqgx $  as above, we define  $ \utildeqgx $  as follows: it is the unital  $ \kqqm $--subalgebra  of  $ \uqgx $  generated by all the elements
  $$  \barK{}_\alpha^{\,\pm 1} := \, \dot{K}_\alpha^{\pm 1}  \;\; ,  \quad
   \barH_\alpha := \big(\, q - 1 \big) \, \dot{H}_\alpha  \;\; ,  \quad
  \barX{}_\alpha^{\,\pm} := (\,q-1) \, \dot{X}_\alpha^\pm  \quad   \eqno  \forall \;\; \alpha \in \intsf(X)  $$
 Similarly, we define  $ \utildeqbxp \, $,  resp.\  $ \utildeqbxm \, $,  in the same way, but discarding the  $ \barX{}_\alpha^{\,-} $'s,  resp.\ the  $ \barX{}_\alpha^{\,-} $'s  ($ \, \alpha \in \intsf(X) $)  from the set of generators; it follows then that  $ \widetilde{\calU}_q\big(\b_X^\pm\big) $  is a subalgebra of  $ \utildeqgx $  too.
\end{definition}

\vskip7pt

   The following is a direct consequence of definitions along with  Proposition \ref{prop: Hopf-str_Uqgx}:

\vskip11pt

\begin{proposition}  \label{prop: pres-utildeqgx}  {\ }
 \vskip3pt
   $ \utildeqgx $  is a Hopf\/  $ \kqqm $--subalgebra  of\/  $ \uqgx \, $,  \,which admits the following presentation: it is the Hopf\/  $ \kqqm $--algebra  with generators  $ \, K_\alpha^{\pm 1} $,  $ \, \barH_\alpha \, $, $ \, \barX^{\,\pm}_\alpha \, \big(\, \alpha \in \intsf(X) \,\big) \, $  and relations
 ($ \, \ia , \ib \in \intsf(X) \, $,  with  $ \, (\ia,\ib) \in \serre{X} \, $  in the last relation)
%
  $$  \displaylines{
   \barK{}_\alpha^{\,\pm 1} \, \barK{}_\beta^{\,\pm 1} = \, \barK{}_\beta^{\,\pm 1} \, \barK{}_\alpha^{\,\pm 1}  \;\; ,  \qquad
 \barK{}_\alpha^{\,\pm 1} \, \barK{}_\beta^{\,\mp 1} = \, \barK{}_\beta^{\,\mp 1} \, \barK{}_\alpha^{\,\pm 1}  \;\; ,  \qquad
 \barK{}_\alpha^{\,\pm 1} \, \barK{}_\alpha^{\,\mp 1} = \, 1  \cr
   \barK{}_{\alpha \oplus \beta}^{\,\pm 1}  \, = \;  \delta_{\alpha \oplus \beta} \, \barK{}_\alpha^{\,\pm 1} \, \barK{}_\beta^{\,\pm 1}  \quad ,  \qquad \qquad
 \barK{}_\alpha  \, = \;  1 \, + \, \barH_\alpha  \cr
   \barH_{\alpha \oplus \beta}  = \,  \delta_{\alpha \oplus \beta} \big(\, \barH_\alpha \, \barK_\beta + \barH_\beta \big)  \quad ,  \quad \qquad
 \barH_\alpha \, \barH_\beta \, = \, \barH_\beta \, \barH_\alpha  \cr
   \barK{}_\alpha^{\,+1} \, \barX{}_\beta^{\,\pm} \, \barK{}_\alpha^{\,-1} = \, q^{\pm(\alpha|\,\beta)} \, \barX{}_\beta^{\,\pm}  \;\; ,   \;\quad
 \barH_\alpha \, \barX{}_\beta^{\,\pm} - \, q^{\pm(\alpha|\,\beta)} \, \barX{}_\beta^{\,\pm} \, \barH_\alpha  \, = \, (\,q - 1) \, {\big(\!\pm(\alpha|\,\beta)\big)}_{\!q} \, \barX{}_\beta^{\,\pm}  \cr
   \barX{}^{\,+}_\ia \, \barX{}^{\,-}_\ib - \barX{}^{\,-}_\ib \, \barX{}^{\,+}_\ia  \,\; = \;\,  \big(\, q - q^{-1} \big) \bigg( \drc{\ia \ib} \Big(\, \barK{}_\ia \! - \! \barK{}_\ia^{\,-1} \Big) \; +   \hfill  \cr
   \hfill   + \;\, \ca{\ia}{\ib} \left(\, q^{\qcc{\ia}{\ib}{+}} \, \barX{}^{\,+}_{\!\ia\ominus\ib} \, \barK{}_\beta^{\,\ca{\ia}{\ib}} \! - q^{\qcc{\ia}{\ib}{-}} \, \barK{}_\alpha^{\,\ca{\ia}{\ib}} \, \barX{}^{\,-}_{\ib\ominus\ia} \right)  \, + \;  \qcb{\ib}{\ia}{} \, q^{\qcb{\ib}{\ia}{}} \, \barX{}^{\,+}_{(\iM{\ia}{\ib})\ominus\ib} \, \barK_{\im{\ia\, }{\, \ib}}^{\,\qcb{\ia}{\ib}{}} \, \barX^{\,-}_{(\iM{\ia}{\ib})\ominus{\ia}} \bigg)  \cr
   \barX{}^{\,\pm}_{\ia} \, \barX{}^{\,\pm}_{\ib} - \, q^{\qcr{\ia}{\ib}{}} \barX{}^{\,\pm}_{\ib} \, \barX{}^{\,\pm}_{\ia}  \; = \; \big(\, q - q^{-1} \big) \, \qcb{\ia}{\ib}{} \, \Big( \pm q^{\qcs{\ia}{\ib}{\pm}} \, \barX{}^{\,\pm}_{\ia\oplus\ib} \, + \, \barX{}^{\,\pm}_{\iM{\ia}{\ib}} \, \barX{}^{\,\pm}_{\im{\ia\, }{\,\ib}} \,\Big)  }  $$
 where we use notation as before,
%
 \,with Hopf structure given by
  $$  \displaylines{
   \varepsilon\big(\, \barH_\alpha \big) \, = \, 0  \quad ,  \quad \qquad  \varepsilon\Big( \barK{}_\alpha^{\,\pm 1} \Big) = 1  \quad ,  \quad \qquad  \epsilon\big(\,\barX{}^{\,+}_\ia\big) = \, 0  \quad ,  \quad \qquad  \epsilon\big(\,\barX{}^{\,-}_\ia\big) = \, 0  \cr
   \Delta\big(\, \barH_\alpha \big) \, = \, \barH_\alpha \otimes \barK{}_\alpha^{\,+1} \! + 1 \otimes \barH_\alpha  \quad ,   \quad \qquad   \Delta\Big( \barK{}_\alpha^{\,\pm 1} \Big) \, = \, \barK{}_\alpha^{\,\pm 1} \!\otimes \barK{}_\alpha^{\,\pm 1}  \cr
   \Delta\big(\,\barX{}^{\,+}_\ia\big)  \; = \;  \barX{}^{\,+}_\ia \otimes 1 + \barK{}_\ia^{\,+1} \otimes \barX{}^{\,+}_\ia \, + \, {\textstyle \sum_{\ia = \ib \oplus \ic}} \, \ca{\ib}{, \,\ia} \; \qcs{\ic}{\ib}{+} \, \barK{}_\ic^{\,+1} \barX{}^{\,+}_\ib \, \otimes \barX{}^{\,+}_\ic  \cr
   \Delta\big(\,\barX{}^{\,-}_\ia\big)  \, = \,  1 \otimes \barX{}^{\,-}_\ia + \barX{}^{\,-}_\ia \otimes \barK{}_\ia^{\,-1} \, - \, {\textstyle \sum_{\ia = \ib \oplus \ic}} \, \ca{\ib}{, \,\ia} \; \qcs{\ic}{\ib}{-} \, \barX{}^{\,-}_\ib \otimes \barX{}^{\,-}_\ic \, \barK{}_\ib^{\,-1}  }  $$
 while the antipode is given by the formula
 $ \; S := \sum_n m^{(n)} \circ {(\id - \iota \circ \epsilon)}^{\otimes n} \circ \Delta^{(n)} \; $.
 \vskip3pt
   Similarly,  $ \utildeqbxp \, $,  resp.\  $ \utildeqbxm \, $,  is a Hopf\/  $ \kqqm $--subalgebra  of  $ \, \uqbxp \, $,  resp.\  of  $ \, \uqbxm \, $   --- hence of  $ \, \utildeqgx $  as well ---   and it admits an analogous presentation by generators and relations.
\end{proposition}

\vskip9pt

   Like for  Proposition \ref{prop: Hopf-str_Uqgx},  the proof of  Proposition \ref{prop: pres-utildeqgx}  above is just a straightforward check.  What is less obvious is that  $ \utildeqgx $  is in fact a  \textsl{Quantum Function Algebra (=QFA)},  but we first have to fix the latter notion.

\vskip11pt

\begin{definition}  \label{def: polyn_QFA}
 We call  \emph{quantized function algebra (=QFA)\/}  any Hopf  $ \kqqm $--algebra  $ F_q $  (in the classical sense) such that  $ \,\; F_1 \, := \, F_q \big/ (\,q-1)\,F_q \;\, $  is a commutative Hopf algebra over  $ \, \kqqm \big/ (\,q-1)\,\kqqm \, \cong \, \Bbbk \; $   --- or, equivalently,  $ F_1 $  is isomorphic to the algebra of functions  $ F[G] $
   of an affine\footnote{\,Note that we are not assuming  $ F_1 $  to be finitely generated, hence the notion of ``affine group'' is meant in its broadest sense.}
 group  $ G \, $;  \,then the formula
 $ \; \displaystyle{ \{x,y\} \, := \, \frac{\; \big[ x' , y' \big] - \big[ y' , x' \big] \;}{q-1} \; \mod (\,q-1)\,F_q } \; $
 --- where  $ \, x' , y' \in F_q \, $  are lifts of  $ \, x, y \in F_1 = F[G] \, $  ---   gives a Poisson bracket in  $ F[G] \, $,  which makes the latter into a Poisson Hopf algebra and so, by general theory,  $ G $  into an affine Poisson group.
                                                                   \par
   In this case, we say that  $ F_q $  is a  \emph{quantization\/}  of the Poisson Hopf algebra  $ F[G] \, $,  or (with a slight abuse of language) of the Poisson group  $ G \, $.
\end{definition}

\vskip11pt

   Using now the language of QFA's, from  Proposition \ref{prop: pres-utildeqgx}  above we get, by straightforward analysis, the following, appropriate version of the QDP for the ``polynomial'' QUEA  $ \uqgx \, $:

\vskip13pt

\begin{theorem}  \label{thm: QDP x Uqgx}
 The Hopf\/  $ \kqqm $--algebra  $ \utildehgx $  is a QFA, which is a quantization of an affine Poisson group  $ G^*_{\!X} $  dual to the Lie bialgebra  $ \g_X \, $.
                                         \par
   More in detail, setting  $ \; \utildeunogx := \utildeqgx \Big/ (\,q-1)\,\utildeqgx \; $,  \,we have:
 \vskip-1pt
   (a)\;  $ \; \utildeunogx \, $  is a  \textsl{commutative}  Hopf algebra;
 \vskip3pt
   (\,b)\;  the Lie bialgebra structure on  $ \, I \big/ I^2 \, $  from  $ \utildeunogx $   --- as in  Definition \ref{def: polyn_QFA}  ---   makes  $ \, I \big/ I^2 \, $  into a Lie  $ \, \Bbbk $--bialgebra  isomorphic to  $ \, \g_X \, $,  via the map (for  $ \, \ia \in \intsf(X) \, $)
  $$  \begin{aligned}
   \Big(\; \barH_\alpha \; \big(\, \text{\rm mod\ } (q-1)\,\utildeunogx \big) \; \text{\rm mod\ } I^2 \,\Big)  &  \,\; \mapsto \;\, \xi_\ia  \\
   \Big(\; \barX{}_\ia^{\,\pm} \; \big(\, \text{\rm mod\ } (q-1)\,\utildeunogx \big) \; \text{\rm mod\ } I^2 \,\Big)  &  \,\; \mapsto \;\, 2 \, x_\ia^\pm
\end{aligned}   \eqno  \forall \; \; \alpha \in \intsf(X)  \quad  $$
\end{theorem}

\begin{proof}
 The situation is entirely similar to that of  Theorem \ref{thm: QDP x Uhgx},  namely claim  \textit{(a)\/}  follows at once from the commutation relations in  Proposition \ref{prop: pres-utildeqgx},  while claim  \textit{(b)\/}  is just a matter of bookkeeping (up to taking into account the different normalizations, which might be somehow misleading).  Indeed, by construction the assignment in  \textit{(b)\/}  yields a  $ \Bbbk $--linear  isomorphism from  $ \, I \big/ I^2 \, $  to  $ \g_X \, $;  \,after this, one has to check, tracking all definitions, that through this map the Lie bracket and Lie cobracket considered in  $ \, I \big/ I^2 \, $  do correspond   --- for the generators of  $ \, I \big/ I^2 \, $  ---   to those in  $ \g_X \, $.  The computations are again very similar to those for  $ \, \utildehgx \big/ \hbar\,\utildehgx \, $,  \,hence we complement what we did for  Theorem \ref{thm: QDP x Uhgx}  considering a different example.
                                                                      \par
   Using notation  $ \, \widetilde{U}_q := \utildeqgx \, $  and  $ \; {\check{X}}_\ic^\pm := \Big(\, \barX{}_\ic^{\,\pm} \, \big(\, \text{\rm mod\ } (\,q-1)\,\widetilde{U}_q \,\big) \,\; \text{\rm mod\ } I^2 \,\Big) \, $,  \,we show that
\begin{equation}  \label{eq: map-Lie-bracket_al-bet}
   \big[\, {\check{X}}_\ia^{\,\pm} \, , {\check{X}}_\ib^{\,\pm} \,\big]  \,\; \mapsto \;\,  \big[\, 2\,x_\ia^{\,\pm} \, , \, 2\,x_\ib^{\,\pm} \,\big]
\end{equation}
 Indeed, setting  $ \; \check{H}_\ia \, := \, \Big(\; \barH_\ia \, \big(\, \text{\rm mod\ } (\,q-1)\,\widetilde{U}_q \,\big) \,\; \text{\rm mod\ } I^2 \,\Big) \, $  we have
%
  \eject

  $$  \displaylines{
   \big[\, {\check{X}}_\ia^{\,\pm} \, , {\check{X}}_\ib^{\,\pm} \,\big]  \; = \;  \Big\{ \Big(\, \barX{}_\ia^{\,\pm} \; \text{\rm mod\ } (\,q-1)\,\widetilde{U}_q \Big) \, , \Big(\, \barX{}_\ib^{\,\pm} \; \text{\rm mod\ } (\,q-1)\,\widetilde{U}_q \Big) \Big\} \; \text{\rm mod\ } I^2  \; =   \hfill  \cr
   = \;  \Bigg(\, \frac{\,\big[\, \barX{}_\ia^{\,\pm} , \barX{}_\ib^{\,\pm} \,\big]\,}{q-1} \; \Big(\, \text{\rm mod\ } (\,q-1)\,\widetilde{U}_q \Big) \Bigg) \; \text{\rm mod\ } I^2  \; =  \cr
   \hfill   = \;  \Bigg(\, \frac{\, \barX{}_\ia^{\,\pm} \, \barX{}_\ib^{\,\pm} - \barX{}_\ib^{\,\pm} \, \barX{}_\ia^{\,\pm} \,}{q-1} \; \Big(\, \text{\rm mod\ } (\,q-1)\,\widetilde{U}_\hbar \Big) \Bigg) \; \text{\rm mod\ } I^2  \; =  \cr
   = \;  \Bigg( \Bigg( {(\,q-1)}^{-1} \bigg( \big(\, q^{\qcr{\ia}{\ib}{}} - 1 \big) \, \barX{}^{\,\pm}_{\ib} \, \barX{}^{\,\pm}_{\ia} \, + \, \big(\, q - q^{-1} \big) \, \qcb{\ia}{\ib}{} \, \Big( \pm q^{\qcs{\ia}{\ib}{\pm}} \barX{}^{\,\pm}_{\ia\oplus\ib} \, + \, \barX{}^{\,\pm}_{\iM{\ia}{\ib}} \, \barX{}^{\,\pm}_{\im{\ia\, }{\,\ib}} \,\Big) \bigg) \Bigg)   \hfill  \cr
   \hfill   \text{\rm mod\ } (\,q-1)\,\widetilde{U}_q \,\Bigg) \;\; \text{\rm mod\ } I^2  \; =  \cr
   = \;  \Bigg( \bigg( {\big(\, {\qcr{\ia}{\ib}{}} \big)}_q \, \barX{}^{\,\pm}_{\ib} \, \barX{}^{\,\pm}_{\ia} \, + \, \big( 1 + q^{-1} \big) \, \qcb{\ia}{\ib}{} \, \Big( \pm q^{\qcs{\ia}{\ib}{\pm}} \barX{}^{\,\pm}_{\ia\oplus\ib} \, + \, \barX{}^{\,\pm}_{\iM{\ia}{\ib}} \, \barX{}^{\,\pm}_{\im{\ia\, }{\,\ib}} \,\Big) \bigg) \;\;   \hfill  \cr
   \hfill   \text{\rm mod\ } (\,q-1)\,\widetilde{U}_q \,\Bigg) \;\; \text{\rm mod\ } I^2  \; =  \cr
   = \;  \bigg( \qcr{\ia}{\ib}{} \, \check{X}^{\,\pm}_{\ib} \, \check{X}^{\,\pm}_{\ia} \, + \, 2 \, \qcb{\ia}{\ib}{} \Big( \pm \check{X}^{\,\pm}_{\ia\oplus\ib} \, + \, \check{X}^{\,\pm}_{\iM{\ia}{\ib}} \, \check{X}^{\,\pm}_{\im{\ia\, }{\,\ib}} \,\Big) \bigg) \;\; \text{\rm mod\ } I^2  \; =   \hfill  \cr
   \hfill   = \;  \Big( \pm 2 \, \qcb{\ia}{\ib}{} \, \check{X}^{\,\pm}_{\ia\oplus\ib} \Big) \; \text{\rm mod\ } I^2  \;\; \mapsto \;\;  \pm \, 4 \, \qcb{\ia}{\ib}{} \, x^{\,\pm}_{\ia\oplus\ib}  \, = \,  4 \, \big[\, x_\ia^\pm \, , \, x_\ib^\pm \,\big]   \, = \,  \big[\, 2\,x_\ia^\pm \, , \, 2\,x_\ib^\pm \,\big]  }  $$
 --- where we wrote  $ \, {(n)}_q := {{\, q^n - 1 \,} \over {\,q-1\,}} = \sum_{s=0}^{n-1} q^s \, $  and  $ \, {(-n)}_q := -q^{-1} {(n)}_{q^{-1}} \, $  for all  $ \, n \in \N \, $  and we took into account that  $ \, \qcb{\ia}{\ib}{} = \ca{\ia}{,\,\ia\oplus\ib}  \, $  ---   so that  \eqref{eq: map-Lie-bracket_al-bet}  is proved.
\end{proof}

\vskip11pt

   The previous analysis has also the following, additional outcome:

\vskip11pt

\begin{corollary}  \label{cor: G Poisson dual}
 The Hopf algebra  $ \; F\big[ G^*_{\!X} \big] \, = \, \utildeunogx \, := \, \utildeqgx \Big/ (\,q-1)\,\utildeqgx \; $  is isomorphic to the Hopf algebra
 $ \; \Bbbk\Big[ {\big\{ X^+_\alpha \, , K^{\pm 1}_\alpha , \, X^-_\alpha \big\}}_{\alpha \in \intsf(X)} \Big] \; $
%
 of polynomials / Laurent polynomials in the variables  $ X^+_\alpha \, $,  $ X^-_\alpha $  (non-invertible) and  $ K^{\pm 1}_\alpha $  (invertible), for all  $ \, \in \intsf(X) \, $.  Thus, the Poisson group  $ G^*_{\!X} \, $,  as an affine variety, is the direct product of an affine space
 $ \, \mathcal{X}^+ \! := \textsl{Spec}\Big( \Bbbk\Big[ {\big\{ X^+_\alpha \big\}}_{\alpha \in \intsf(X)} \Big] \Big) $,
 an affine torus
 $ \, \mathcal{K} := \textsl{Spec}\Big( \Bbbk\Big[ {\big\{ K^{\pm 1}_\alpha \big\}}_{\alpha \in \intsf(X)} \Big] \Big) $,
 and another affine space
 $ \, \mathcal{X}^- \! := \textsl{Spec}\Big( \Bbbk\Big[ {\big\{ X^-_\alpha \big\}}_{\alpha \in \intsf(X)} \Big] \Big) $,
 all having infinite dimension equal to  $ \, \big| \intsf(X) \big| \, $,  \,in short  $ \; G^*_{\!X} \, \cong \, \mathcal{X}^{+\!} \times \mathcal{K} \times \mathcal{X}^- \; $.
                                                                             \par
   Finally, the subvarieties  $ \, \mathcal{X}^{+\!} \times \mathcal{K} \, $,  $ \, \mathcal{K} \, $  and  $ \, \mathcal{K} \times \mathcal{X}^- \, $  are Poisson subgroups of  $ \, G^*_{\!X} \, $.
\end{corollary}

\begin{proof}
 The claim follows from  Theorem \ref{thm: QDP x Uqgx}  and the presentation of  $ \, F\big[ G^*_{\!X} \big] = \utildeunogx \, $  that one gets from the presentation of  $ \utildeqgx $  in  Proposition \ref{prop: pres-utildeqgx}.
\end{proof}

\vskip9pt

\begin{Observation}  \label{obs: utildeqgx over Zqqm}
 Much like we did for  $ \uqgx $  in  Observation \ref{obs: uqgx over Zqqm},  it is worth remarking that in the formulas occurring in  Proposition \ref{prop: pres-utildeqgx}  \textsl{all coefficients that show up actually belong to the (sub)ring\/}  $ \Z\big[\,q,q^{-1}\big] \, $.  Thus the Hopf algebra  $ \utildeqgx $  can actually be defined  \textsl{over\/}  $ \Z\big[\,q,q^{-1}\big] \, $,  \,hence one can consider  $ R $--integral  forms of it over any ring  $ R \, $,  \,hence possibly look for specializations at roots of unity, etc., much like in the study of quantum groups by Kac, De Concini, Procesi and others, including   --- in an infinite dimensional setting, such as is that of  $ \g_X $  ---   the study in \cite{beck-94,beck-96}  for the case of quantum affine Kac--Moody algebras.
\end{Observation}

\medskip

\subsection{Polynomial QDP for  $ \uqgx \, $:  the intrinsic recipe}
 In  \S \ref{subsec: QDP-polyn}  above we introduced a  $ R_q $--subalgebra  $ \utildeqgx $  of  $ \uqgx \, $,  we proved that it is a Hopf  $ R_q $--algebra,  and even a QFA that is a quantization of  $ G_X^* \, $,  a well-defined affine Poisson group dual to  $ \g_X \, $.  In this sense, we have found a concrete realization of the QDP for the polynomial QUEA  $ \uqgx \, $.
                                                            \par
   Note that the very definition of  $ \uqgx $  strictly mimics the similar constructions for the standard polynomial QUEA  $ U_q(\g) $  associated (following Drinfeld and Jimbo) with a Kac-Moody algebra  $ \g $  of finite or affine type  (cf.\ \cite{gavarini_PJM-98,gavarini-00},  and references therein).  In fact, the present framework is even more convenient as we already have in  $ \uqgx \, $,  \,by definition, some built-in ``quantum root vectors''  $ \dot{X}{}_\ia^\pm $  that lift (up to a factor 2) the ``root vectors''  $ x_\ia^\pm $  in  $ \g_X $   --- whereas in the case of finite or affine Kac-Moody  $ \g $  one has to construct them out of ``simple root vectors''.
                                                            \par
   We shall now show that, like we did with  $ \utildehgx $  for the ``non-polynomial version of the QDP, we can provide for  $ \utildeqgx $  an intrinsic description.

\vskip11pt

\begin{definition}  \label{def: U'qgx}
 For every  $ \, n \in \N \, $,  let  $ \, \delta_n := (\id - \iota \circ \epsilon) \circ \Delta^{(n)} \, $.  Then for  $ \uqgx $  as in  Definition \ref{def: polyn_Uqgx},  we set
 \vskip5pt
   \hfill   $  {\uqgx}'  \; := \;  \Big\{\, \eta \in \uqgx \,\Big|\; \delta_n(\eta) \in {(\,q-1)}^n\,{\uqgx}^{\otimes n} \;\; \forall \; n \in \N \,\Big\}  $   \hfill
\end{definition}

\vskip3pt

   Indeed, the above just applies the general definition of the Drinfeld's functor  $ \, H \mapsto H' \, $  for Hopf algebras  $ H $  in much larger generality (than QUEA's alone), which is detailed in  \cite[\S 2.1]{gavarini-07}  extending Drinfeld's original idea.

\vskip3pt

   The key point now is our next result:

\vskip11pt

\begin{theorem}  \label{thm: uqgx'=utildeqgx}
 With notation as before, we have  $ \,\; \utildeqgx \, = \, {\uqgx}' \;\, $.
\end{theorem}

\begin{proof}
 The proof adapts the arguments used for  Theorem \ref{thm: uhgx'=utildehgx},  up to technicalities.
 \vskip5pt
   By the properties of the maps  $ \delta_n \, (\, n \in \N \,) \, $  explained in  \cite[Lemma 3.2]{gavarini-07},  one gets easily that  $ \, {\uqgx}' \, $  is a unital  $ \kqqm $--subalgebra  of  $ \uqgx \, $.  Moreover, a direct check shows that  $ \; \barH_\ia \, , \barX{}_\ia^\pm \in {\uqgx}' \; $  for all  $ \, \ia \in \intsf(X) \, $;  \,but then it is also  $ \, \barK{}_\alpha^{+1} = 1 + \barH_\alpha \, \in \, {\uqgx}' \, $.  On the other hand, we have
  $$  \barK{}_\alpha^{\,+1} \cdot {\textstyle \sum_{n=0}^{N-1}} {(-1)}^n \barH{}_\alpha^{\,n}  \; = \;  1 - {(-1)}^N \barH{}_\alpha^{\,N}  \; \in \;  1 + {(\,q-1)}^N \, \uqgx   \eqno \forall \;\; N \in \N_+  \quad  $$
 and then multiplying by  $ \, K_\alpha^{\,+1} = \barK{}_\alpha^{\,+1} \, $  this yields
  $$  \displaylines{
   \quad   {\textstyle \sum_{n=0}^{N-1}} {(-1)}^n \barH{}_\alpha^{\,n}  \; = \;  K_\alpha^{\,-1} - {(-1)}^N \barH{}_\alpha^{\,N} \, \barK{}_\alpha^{\,-1}  \quad \Longrightarrow   \hfill  \cr
   \hfill   \Longrightarrow \quad  K_\alpha^{\,-1} \, = \, {\textstyle \sum_{n=0}^{N-1}} {(-1)}^n \barH{}_\alpha^{\,n} + {(-1)}^N \barH{}_\alpha^{\,N} \, \barK{}_\alpha^{\,-1}  \, = \;  \eta_N + \chi_N   \quad }  $$
 for all  $ \, N \in \N_+ \, $,  \,where  $ \, \eta_N \in {\uqgx}' \, $  and  $ \, \chi_N \in {(\,q-1)}^N \uqgx \, $.  But then
  $$  \delta_N\big( K_\alpha^{\,-1} \big)  \; = \;  \delta_N(\eta_N) + \delta_N(\chi_N)  \; \in \;  {(\,q-1)}^N \uqgx^{\otimes N}   \eqno \forall \;\; N \in \N  $$
 whence  $ \, K_\alpha^{\,-1} \in {\uqgx}' \, $  as well  ($ \, \ia \in \intsf(X) \, $).  Thus all the generators of  $ \utildeqgx $  belong to  $ {\uqgx}' $  and the latter is a  $ \kqqm $--algebra,  hence  $ \; \utildehgx \subseteq {\uhgx}' \; $.
 \vskip5pt
   As to the converse inclusion, we need first to lay the groundwork.
                                                             \par
   First of all, let us denote by  $ \, \uqnxp \, $,  resp.\  $ \, \uqfx \, $,  resp.\  $ \, \uqnxm \, $,  \,the unital  $ \kqqm $--subalgebra  of  $ \uqgx $  generated by the  $ \dot{X}{}^+_\alpha $'s,  \,resp.\ the  $ \dot{K}{}^{\pm 1}_\alpha $'s  and the  $ \dot{H}_\alpha $'s,  \,resp.\ the  $ \dot{X}{}^-_\alpha $'s.  Then, by a standard argument, one easily deduces from the relations in the presentation of  $ \uqgx $   --- cf.\ Proposition \ref{prop: Hopf-str_Uqgx}  ---   that the latter admits the ``triangular decomposition''
\begin{equation}  \label{eq: triang-decomp_Uqgx}
   \uqgx  \;\; \cong \;\;  \uqnxp \otimes \uqfx \otimes \uqnxm
\end{equation}
   \indent   As a second step, let us fix any total order on  $ \intsf(X) \, $.  Then, from the relations among all the  $ \dot{X}_\alpha^{\,\pm} $'s,  one easily finds (again by standard arguments) that the set of ordered monomials
%
%
 $ \; \Big\{\, {\overrightarrow{\prod}_\ia} {\big( \dot{X}_\ia^{\pm} \big)}^{n_\ia^\pm} \,\Big|\; n_\ia^\pm \in \N \, , \; \forall \; \ia \in \intsf(X) \, , \textsl{\ a.\ a.\ } 0 \Big\} \; $
 --- where  ``$ \, \overrightarrow\prod \, $''  denotes an ordered product, and the  $ \ia $'s  range within  $ \intsf(X) \, $,  \,and again  ``$ \, \textsl{\ a.\ a.\ 0} \, $''  stands for ``\,almost all 0\,'' ---   is a PBW-type spanning set of  $ \, \calU_q\big(\hskip1pt\n^\pm_X\big) $  over  $ \, \kqqm \, $   --- in fact, it is even a  $ \kqqm $--basis.  Similarly, from the relations among all the  $ \dot{H}_\alpha $'s  and all the  $ \dot{K}{}^{\pm 1}_\gamma $'s  one finds that the set of (ordered) products
 $ \; \Big\{\, {\overrightarrow{\prod}_\ia} {\big( \dot{H}_\ia \big)}^{e^+_\ia} {\big( \dot{K}_\ia^{-1} \big)}^{n^-_\ia} \,\Big|\; e_\ia^\pm \in \N \, , \; 0 \in \big\{ e_\ia^+ , e_\ia^- \big\} \, , \; \forall \; \ia \in \intsf(X) \, , \textsl{\ a.\ a.\ } 0 \Big\} \; $
 is a PBW-type spanning set of  $ \, \uqfx $  over  $ \, \kqqm \, $.  In the end, this together with  \eqref{eq: triang-decomp_Uqgx}  tells us that the set of all ordered monomials of the form
\begin{equation}  \label{eq: PBW-spanset x uqgx}
  {\textstyle \overrightarrow{\prod\limits_\ia}} {\big( \dot{X}_\ia^+ \big)}^{n_\ia^+} \cdot \,
   {\textstyle \overrightarrow{\prod\limits_\ia}} {\big( \dot{H}_\ia \big)}^{e^+_\ia} \! {\big( \dot{K}_\ia^{-1} \big)}^{e^-_\ia} \cdot \,
   {\textstyle \overrightarrow{\prod\limits_\ia}} {\big( \dot{X}_\ia^- \big)}^{n_\ia^-}
\end{equation}
 is a  $ \kqqm $--spanning  set for  $ \uqgx \, $.  On the other hand, we recall   --- cf.\ the proof of  Theorem \ref{thm: uhgx'=utildehgx}  ---   that, at the semiclassical level), the set of ordered monomials
\begin{equation}   \label{eq: span-set x ugx}
  \bigg\{\; {\textstyle \overrightarrow{\prod\limits_\ia}} {\big( x_\ia^+ \big)}^{n_\ia^+} \, {\textstyle \overrightarrow{\prod\limits_\ia}} \, \xi_\ia^{\,e_\ia} \, {\textstyle \overrightarrow{\prod\limits_\ia}} {\big( x_\ia^- \big)}^{n_\ia^-} \;\bigg|\;\, n_\ia^+ \, , e_\ia \, , n_\ia^- \in \N \, , \,\; \forall \; \ia \in \intsf(X) \, , \textsl{\ a.\ a.\ } 0 \;\bigg\}
\end{equation}
 is a PBW-type  $ \Bbbk $--spanning  set for  $ U(\g_X) \, $.  Note also that, through the isomorphism in  Theorem \ref{thm: specializ-Uqgx},  every PBW monomial in  \eqref{eq: PBW-spanset x uqgx}  maps onto a corresponding monomial in  \eqref{eq: span-set x ugx}  multiplied by a suitable power of 2\,.
 \vskip5pt
   Now let  $ \, \eta \in {\uqgx}' \, $.  Using the spanning set in  \eqref{eq: PBW-spanset x uqgx},  we can expand  $ \eta $  as
\begin{equation}  \label{eq: (q-1)-expansion x eta}
   \eta \; = \, {\textstyle \sum_{\ell=0}^N} \, {(\,q-1)}^\ell \, \eta_\ell   \;\quad  \text{with}  \quad  \eta_\ell \in \big( \uqgx \setminus (\,q-1)\,\uqgx \big) \cup \{0\}   \quad \forall \;\; \ell
\end{equation}
 in such a way that each non-zero  $ \eta_\ell $  in  \eqref{eq: (q-1)-expansion x eta}  in turn expands as a linear  combination of quantum PBW monomials from  \eqref{eq: PBW-spanset x uqgx},  say
\begin{equation}  \label{eq: qPBW-expansion_polyn}
  \eta_\ell  \;\; = \;\hskip-3pt  \sum_{\underline{n}^+ \!, \, \underline{e}^+ , \, \underline{e}^- , \, \underline{n}^-} \hskip-9pt
  \kappa_{\ell\,;\,\underline{n}^-\!,\,\underline{e}^-}^{\,\underline{e}^+\!,\,\underline{e}^+} \cdot
   {\textstyle \overrightarrow{\prod\limits_\ia}} {\big( \dot{X}_\ia^+ \big)}^{n_\ia^+} \cdot \,
   {\textstyle \overrightarrow{\prod\limits_\ia}} {\big( \dot{H}_\ia \big)}^{e^+_\ia} \! {\big( \dot{K}_\ia^{-1} \big)}^{e^-_\ia} \cdot \,
   {\textstyle \overrightarrow{\prod\limits_\ia}} {\big( \dot{X}_\ia^- \big)}^{n_\ia^-}
\end{equation}
 with  $ \; \kappa_{\ell\,;\,\underline{n}^-\!,\,\underline{e}^-}^{\,\underline{e}^+\!,\,\underline{e}^+}
\in \kqqm \setminus (\,q-1)\,\kqqm \; $   --- where  $ \underline{n}^\pm $,  resp.\  $ \underline{e}^\pm \, $,  is the string of the  $ n_\ia^\pm $'s,  resp.\ of the  $ e_\ia^\pm $'s.
 \vskip3pt
   If  $ \; {(\,q-1)}^\ell \, \eta_\ell \in \utildeqgx \; $  for all  $ \, \ell \in \{0\,,\dots,N\} \, $,  \,then  $ \, \eta \in \utildeqgx \, $   --- by the first part of the proof ---   and we are done.  Otherwise, let  $ \, \ell_0 \in \{0\,,\dots,N\} \, $  be the smallest index such that  $ \; {(\,q-1)}^{\,\ell_0} \, \eta_{\ell_0} \not\in \utildeqgx \; $;  \,then
  $$  \eta_+  \; := \;  \eta - {\textstyle \sum_{\ell = 0}^{\ell_0 - 1}} {(\,q-1)}^{\,\ell} \, \eta_\ell  \; = \;  {\textstyle \sum_{\ell = \ell_0}^N} {(\,q-1)}^{\,\ell} \, \eta_\ell  \,\; \in \;\,  {\uqgx}'  $$
 By construction we have  $ \, \eta_{\ell_0} \in \uqgx \setminus {(\,q-1)} \, \uqgx \, $,  \,thus the coset
  $$  \overline{\eta}_{\ell_0}  \; := \;  \big(\, \eta_{\ell_0} \!\! \mod {(\,q-1)}\,\uqgx \,\big)  $$
 in  $ \; \uqgx \Big/ {(\,q-1)}\,\uqgx \, = \, U(\g_X) \; $  is non-zero, hence it has a well-defined degree  $ \, \partial\big( \overline{\eta}_{\ell_0} \big) \, $  with respect to the standard filtration in  $ U(\g_X) \, $.  But then  \cite[Lemma 4.12]{etingof-kazhdan-96}  gives  $ \, \partial\big( \overline{\eta}_{\ell_0} \big) \leq \ell_0 \, $:  \,therefore, the cosets of all quantum PBW monomials occurring (with non-zero coefficients) in  \eqref{eq: qPBW-expansion_polyn}  when  $ \, \ell = \ell_0 \, $  are (semiclassical) PBW monomials from  \eqref{eq: span-set x ugx}  that have degree bounded by  $ \, \ell_0 \, $.  Since the coset of the quantum PBW monomial
 $ \; {\textstyle \overrightarrow{\prod\limits_\ia}} {\big( \dot{X}_\ia^+ \big)}^{n_\ia^+} \cdot \,
   {\textstyle \overrightarrow{\prod\limits_\ia}} {\big( \dot{H}_\ia \big)}^{e^+_\ia} \! {\big( \dot{K}_\ia^{-1} \big)}^{e^-_\ia} \cdot \,
   {\textstyle \overrightarrow{\prod\limits_\ia}} {\big( \dot{X}_\ia^- \big)}^{n_\ia^-} \; $
 is the rescaled (semiclassical) PBW monomial
 $ \; 2^{|\,\underline{n}^+| + |\,\underline{e}^+| + |\,\underline{n}^-|} \, {\textstyle \overrightarrow{\prod\limits_\ia}} {\big( x_\ia^+ \big)}^{n_\ia^+} \cdot \,
   {\textstyle \overrightarrow{\prod\limits_\ia}} {\big( \xi_\ia \big)}^{e^+_\ia} \cdot \,
   {\textstyle \overrightarrow{\prod\limits_\ia}} {\big( x_\ia^- \big)}^{n_\ia^-} \; $
 --- where by  $ \, |\,\underline{s}| \, $  we denote the sum of all elements in any string  $ \, \underline{s} = {(s_\ia)}_{\ia \in \intsf(X)} \, $  with almost all entries being zero ---   we conclude that this bound on the degree reads
  $$  |\,\underline{n}^+| + |\,\underline{e}^+| + |\,\underline{n}^-|  \; \leq \;  \ell_0   \qquad \qquad  \forall \quad  \kappa_{\ell\,;\,\underline{n}^-\!,\,\underline{e}^-}^{\,\underline{e}^+\!,\,\underline{e}^+} \not= \, 0  $$
 But this implies
  $$  \displaylines{
   {(\,q-1)}^{\ell_0} \, \eta_{\ell_0}  \; = \hskip-7pt  \sum_{\underline{n}^+ \!, \, \underline{e}^+ , \, \underline{e}^- , \, \underline{n}^-} \hskip-13pt
 \kappa_{\ell\,;\,\underline{n}^-\!,\,\underline{e}^-}^{\,\underline{e}^+\!,\,\underline{e}^+} \, {(\,q-1)}^{\ell_0} \cdot
  {\textstyle \overrightarrow{\prod\limits_\ia}} {\big( \dot{X}_\ia^+ \big)}^{n_\ia^+} \,
  {\textstyle \overrightarrow{\prod\limits_\ia}} {\big( \dot{H}_\ia \big)}^{e^+_\ia} \! {\big( \dot{K}_\ia^{-1} \big)}^{e^-_\ia} \,
  {\textstyle \overrightarrow{\prod\limits_\ia}} {\big( \dot{X}_\ia^- \big)}^{n_\ia^-}  \; =  \cr
   = \hskip-7pt  \sum_{\underline{n}^+ \!, \, \underline{e}^+ , \, \underline{e}^- , \, \underline{n}^-} \hskip-13pt
 \kappa_{\ell_0;\,\underline{n}^-\!,\,\underline{e}^-}^{\,\underline{e}^+\!,\,\underline{e}^+} \, {(\,q-1)}^{\ell_0 - |\,\underline{n}^+| + |\,\underline{e}^+| + |\,\underline{n}^-|} \cdot
  {\textstyle \overrightarrow{\prod\limits_\ia}} {\big( \bar{X}_\ia^+ \big)}^{n_\ia^+} \,
  {\textstyle \overrightarrow{\prod\limits_\ia}} {\big( \bar{H}_\ia \big)}^{e^+_\ia} \! {\big( \bar{K}_\ia^{-1} \big)}^{e^-_\ia} \,
  {\textstyle \overrightarrow{\prod\limits_\ia}} {\big( \bar{X}_\ia^- \big)}^{n_\ia^-}  }  $$
 which means  $ \; {(\,q-1)}^{\ell_0} \, \eta_{\ell_0} \in \utildeqgx \, $   --- by the first part of the proof ---   against the assumption  $ \; {(\,q-1)}^{\ell_0} \, \eta_{\ell_0} \not\in \utildeqgx \; $,  \,a contradiction.
\end{proof}

\bigskip
 \bigskip

\end{document}